\documentclass[11pt,a4paper]{article}
\title{\bf Sharp Heat Kernel Bounds and Entropy in Metric Measure Spaces}
\author{Huaiqian Li\footnote{Email: hqlee@scu.edu.cn. Partially supported by the National Natural Science Foundation of China (NSFC) No.11401403 and the Australian Research Council (ARC) grant DP130101302. }\vspace{3mm}\\
{\footnotesize School of Mathematics, Sichuan University, Chengdu 610064, P. R. China}
}
\date{}
\usepackage{amssymb,amsmath,amsfonts,amsthm,color,mathrsfs}
\usepackage{mathpazo}

\def\R{\mathbb{R}}

\def\L{\mathcal{L}}

\def\d{\textup{d}}

\def\CD{\textup{CD}}

\def\csch{\textup{csch}}

\def\RCD{\textup{RCD}}
\def\<{\langle}
\def\>{\rangle}
\def\Proof.{\noindent{\bf Proof. }}

\def\loc{\textup{loc}}

\def\newdot{{\kern.8pt\cdot\kern.8pt}}

\newtheorem{theorem}{Theorem}[section]
\newtheorem{lemma}[theorem]{Lemma}
\newtheorem{corollary}[theorem]{Corollary}

\newtheorem{definition}[theorem]{Definition}
\theoremstyle{definition}\newtheorem{remark}[theorem]{Remark}

\begin{document}

\maketitle
\makeatletter % '@' is now a normal "letter" for TeX
\renewcommand\theequation{\thesection.\arabic{equation}}
\@addtoreset{equation}{section}
\makeatother % '@' is restored as a "non-letter" character for TeX

\begin{abstract}
We establish sharp upper and lower bounds of Gaussian type for the heat kernel in the metric measure space satisfying $\RCD(0,N)$ ( equivalently, $\RCD^\ast(0,N)$) condition with $N\in \mathbb{N}\setminus\{1\}$ and having maximum volume growth, and then show its application on the large-time asymptotics of the heat kernel, sharp bounds on the (minimal) Green function, and above all, the large-time asymptotics of the Perelman entropy and the Nash entropy, where for the former the monotonicity of the Perelman entropy is proved. The results generalize the corresponding ones in Riemannian manifolds and also in metric measure spaces obtained recently by the author with R. Jiang and H. Zhang in \cite{JLZ2014}.
\end{abstract}



\section{Introduction}\hskip\parindent
Let $K\in \R$ and $N\in [1,\infty)$. In the pioneer works of Lott--Villani \cite{LottVillani2009} and Sturm  \cite{Sturm2006a,Sturm2006b}, a notion of Ricci curvature bounded from below by $K$ and dimension bounded above by $N$ in the metric measure space $(X,d,\mu)$, called the curvature-dimension condition and denoted by $\CD(K,N)$, was proposed independently by the aforementioned authors (note that only the cases $\CD(0,N)$ and $\CD(K,\infty)$ are considered in \cite{LottVillani2009}).  A lot of work on the study of functional and geometric implications in the $\CD(K,N)$ space has been done since then; refer to \cite[Part III]{Villani2009} for an elaborate presentation of the theory. Recently, Ambrosio--Gigli--Savar\'{e} \cite{AmbrosioGigliSavare2011b} introduced the Riemannian curvature condition, denoted by $\RCD(K,\infty)$, which is stronger than the curvature-dimension condition $\CD(K,\infty)$ in the sense by requiring additionally the space to be infinitesimally Hilbertian, and established the equivalence between the $\RCD(K,\infty)$ condition and the curvature-dimension condition in the sense of Bakry--Emery \cite{BakryEmery} (see \cite{AmbrosioGigliSavare2012}). Erbar--Kuwada--Sturm \cite{eks2013} introduced the Riemannian curvature-dimension condition with $N$ finite, denoted by $\RCD^\ast(K,N)$ (see also \cite{AMS2012}), which is a strengthening of the reduced curvature-dimension condition $\CD^\ast(K,N)$ introduced in \cite{BacherSturm2010}.

Let $(X,d,\mu)$ be an $\RCD^\ast(K,N)$ space with $K\leq 0$ and $N\in(1,\infty)$. In a recent joint work \cite{JLZ2014}, by using the comparison result (see e.g. Lemma \ref{kernelcomparison} below), the author with R. Jiang and H. Zhang established the following heat kernel upper and lower bounds of Gaussian type. More precisely, if $K=0$, then given any $\epsilon>0$, there exists a constant $C_1(\epsilon)>0$ such that
\begin{equation*}
\frac{C_1(\epsilon)^{-1}}{\mu(B(y,\sqrt t))}\exp\left\{-\frac{d^2(x,y)}{(4-\epsilon)t}\right\}\le p_t(x,y)\le \frac{C_1(\epsilon)}{\mu(B(y,\sqrt t))}\exp\left\{-\frac{d^2(x,y)}{(4+\epsilon)t}\right\},
\end{equation*}
for all $t>0$ and all $x,y\in X$; if $K<0$, then given any $\epsilon>0$, there exist constants $C_2(\epsilon),C_3(\epsilon)>0$ such that
\begin{eqnarray*}
\frac{C_2(\epsilon)^{-1}}{\mu(B(y,\sqrt t))}\exp\left\{-\frac{d^2(x,y)}{(4-\epsilon)t}-C_3(\epsilon) t\right\}&\le& p_t(x,y)\\
&\le&\frac{C_2(\epsilon)}{\mu(B(y,\sqrt t))}\exp\left\{-\frac{d^2(x,y)}{(4+\epsilon)t}+C_3(\epsilon) t\right\},
\end{eqnarray*}
for all $t>0$ and all $x,y\in X$.

In this note, we show more explicit and sharper upper and lower bounds of Gaussian type for the heat kernel in the $\RCD(0,N)$ space (equivalently, $\RCD^\ast(0,N)$ space) with $N\in (1,\infty)$ by assuming additionally that $N$ is an integer and the space has maximum volume growth, which is also a generalization of the result established in the Riemannian manifold (see \cite{LiTamWang}). And then we show some applications.

In what follows, we give a short introduction of the $\RCD(K,N)$ space and present some known results in Section 2. In Section 3, we establish the sharp heat kernel lower and upper bounds. Finally, in Section 4, we show the large-time asymptotics of the Perelman entropy and the Nash entropy, where for the former, we prove the monotonicity of the Perelman entropy, and for the later, it is a direct application of our sharp heat kernel bounds.

\section{Preliminaries}
\hskip\parindent In this section, we briefly recall some basic notions and several auxiliary results. More details can be found in \cite{AmbrosioGigliSavare2011b,AmbrosioGigliSavare2014,agmr2015,gi2012}.
\subsection{Sobolev spaces and the Laplacian}
\hskip\parindent
Let $(X,d)$ be a complete and separable metric space and let $C([0, 1],X)$ be the space of continuous curves on $[0, 1]$ with values in $X$ equipped with the supremum norm. For $t\in [0, 1]$, the map $e_t :C([0, 1],X) \rightarrow X$ is the evaluation at
time $t$ defined by
$$e_t(\gamma):=\gamma_t.$$

A curve $\gamma: [0,1] \rightarrow X$ is in the absolutely continuous
class $AC^q([0,1],X)$ for some $q\in [1,\infty]$, if there exists $f\in L^q([0,1])$ such that, \begin{eqnarray}\label{ac}
d(\gamma_s,\gamma_t)\leq \int_s^t g(r)\,d r,\quad\mbox{for any }s,t\in [0,1]\mbox{ satisfying } s<t.
\end{eqnarray}
It is true that, if $\gamma\in AC^p([0,1];X)$, then the metric slope
$$\lim_{\delta\rightarrow 0}\frac{d(\gamma_{r+\delta},\gamma_r)}{|\delta|},$$
denoted by $|\dot{\gamma}_r|$, exists for $\mathcal{L}^1$-a.e. $r\in [0,1]$, belongs to $L^p([0,1])$, and it is the minimal function $g$ such that \eqref{ac} holds (see Theorem 1.1.2 in \cite{AmbrosioGigliSavare2005}). The length of the absolutely continuous curve $\gamma: [0,1]\rightarrow X$ is denoted by $\int_0^1 |\dot{\gamma}_r|\,d r$. We call that $(X,d)$ is a length space if
$$d(x_0,x_1)=\inf\left\{\int_0^1 |\dot{\gamma}_r|\,d r:\, \gamma\in AC^1([0,1],X),\, \gamma_i=x_i,\,i=0,1 \right\},\quad \forall\, x_0,x_1\in X.$$

Let $\mu$ be a $\sigma$-finite Radon measure on $(X,d)$ with support the whole space $X$. Throughout the work, we call the triple $(X,d,\mu)$ the metric measure space.

\begin{definition}[Test Plan]
Let $\pi$ be a probability measure on $C([0, 1],X)$. We say that $\pi$ is a test plan if there exists a constant $C>0$ such that
$$(e_t)_\sharp{ \pi}\le C\mu,\quad\mbox{for all }t\in [0,1],$$
and
$$\int \int_0^1|\dot{\gamma}_t|^2\,\d t\,\d\pi(\gamma)<\infty.$$
  \end{definition}

\begin{definition}[Sobolev class] \label{sobolev}
The Sobolev class $S^2(X)$ (resp. $S_{\loc}^2(X)$) is the space of all Borel functions $f: X\rightarrow \R$, for which there exists a non-negative function $G\in L^2(X)$ (resp. $G\in L^2_{\loc}(X)$) such that, for each test plan $\pi$, it holds
\begin{equation}\label{curve-sobolev}
\int |f(\gamma_1)-f(\gamma_0)|\,\d\pi(\gamma)\leq \int \int_0^1 G(\gamma_t)|\dot{\gamma}_t|\,dt\,\d\pi(\gamma).
\end{equation}
\end{definition}
It then follows from a compactness argument that,  for each  $f\in S^2(X)$ there exists a unique minimal $G$
in the $\mu$-a.e. sense such that \eqref{curve-sobolev} holds. We then denote the minimal $G$ by $|\nabla  f|_w$
and call it the minimal weak upper gradient following \cite{AGS2013}.

The inhomogeneous Sobolev space $W^{1,2}(X)$ is defined as  $S^2(X)\cap L^2(X)$, which equipped with the norm
$$\|f\|_{W^{1,2}(X)}:=\Big(\|f\|_{L^2}^2+\| |\nabla  f|_w\|_{L^2(X)}^2\Big)^{1/2},$$
is a Banach space, but not a Hilbert space in general.

The local Sobolev space $W^{1,2}_{\mathrm{loc}}(\Omega)$ for an open set $\Omega\subset X$,  and the Sobolev space with compact support
$W^{1,2}_{c}(X)$ can be defined in an obvious manner. See \cite{AGS2013,ch,sh} for the study of relevant Sobolev spaces.

The following definitions and results are mainly borrowed from \cite{gi2012}.
\begin{definition}[Infinitesimally Hilbertian Space] Let $(X, d,\mu)$ be a metric measure
space. If $W^{1,2}(X)$ is a Hilbert space, then we call that $(X, d,\mu)$ is an infinitesimally Hilbertian space.
\end{definition}

Notice that, from the definition, it follows that $(X, d,\mu)$ is infinitesimally Hilbertian if and only if,
for any $f,g\in S^2(X)$, it holds
$$\||\nabla (f+g)|_w\|_{L^2(X)}^2+\||\nabla (f-g)|_w\|_{L^2(X)}^2=2\Big(\||\nabla f|_w\|_{L^2(X)}^2+\||\nabla g|_w\|_{L^2(X)}^2\Big).$$

\begin{definition} Let $(X, d,\mu)$ be an infinitesimally Hilbertian space, $\Omega$ be an open subset of $X$ and $f, g\in S^2_{\mathrm{loc}}(\Omega)$.
The map $\langle \nabla f, \nabla g\rangle :\, \Omega \rightarrow \R$ is defined as
$$\langle \nabla f, \nabla g\rangle(x):= \inf_{\epsilon>0} \frac{|\nabla (g+\epsilon f)|_w^2(x)-|\nabla g|_w^2(x)}{2\epsilon},\quad\mbox{for }\mu\mbox{-a.e. }x\in X,$$
where the infimum is intended as $\mu$-essential infimum.
\end{definition}

The inner product $\langle \nabla f, \nabla g\rangle$ is linear and satisfies the Cauchy--Schwarz inequality,
the chain rule and the Leibniz rule (see e.g. \cite{gi2012}).

With the aid of the inner product, we can define the Laplacian operator as below.
\begin{definition}[Laplacian] \label{glap}
Let $(X, d,\mu)$ be an infinitesimally Hilbertian space and let $f\in W^{1,2}_{\mathrm{loc}}(X)$. We call $f\in {\mathcal D}_{\mathrm{loc}}(\Delta)$, if there exists $h\in L^1_{\mathrm{loc}}(X)$ such that, for each $\psi\in W^{1,2}_{c}(X)$, it holds that
\begin{equation*}
\int_X\langle \nabla f,\nabla\psi\rangle\, \d\mu=-\int_X h\psi\,\d\mu.
\end{equation*}
We denote $h$ as $\Delta f$ and call it the Laplacian of $f$. If $f\in W^{1,2}(X)$ and $h\in L^2(X)$, then we write $f\in {\mathcal D}(\Delta)$.
\end{definition}

Notice that the Laplacian operator is linear due to that $(X, d,\mu)$ is infinitesimally Hilbertian. From the Leibniz rule of the inner product, it follows that if $f,g\in {\mathcal D}_{\mathrm{loc}}(\Delta)\cap L^\infty_{\mathrm{loc}}(X)$ (resp. Lipschitz continuous functions $f,g\in {\mathcal D}(\Delta)\cap L^\infty_{\mathrm{loc}}(X)$), then $fg\in {\mathcal D}_{\mathrm{loc}}(\Delta)$ (resp. $fg\in {\mathcal D}(\Delta)$) satisfies
$\Delta (fg)=g\Delta f+f\Delta g+2\nabla f\cdot\nabla g$.

\subsection{Curvature-dimension conditions and consequences}
\hskip\parindent Let $(X,d,\mu)$ be  an infinitesimally Hilbertian space. Then the heat flow $\{e^{t\Delta}\}_{t\geq0}$ is linear. Denote by $\{P_t\}_{t\geq0}$  the heat semigroup corresponding to the Dirichlet form $\big(\mathcal{E},W^{1,2}(X)\big)$ defined by
$$\mathcal{E}(f,g)=\int_X \langle \nabla f,\nabla g \rangle\, \d\mu,\quad f,g\in W^{1,2}(X).$$
Moreover, in the $\RCD(K,N)$ space with $K\in\R$ and $N\in (1,\infty)$, introduced in Definition \ref{rcd} below, $P_t=e^{t\Delta}$ for all $t\geq0$.

Now we recall the definition of the $\RCD(K,N)$ space. Let $\mathcal{P}(X)$ be the set of all the Borel probability measures on $X$, and let $\mathcal{P}(X,\mu)$ be the subset of $\mu$-absolutely continuous measures in $\mathcal{P}(X)$. Given two numbers $K\in \R$ and $N\in (1,\infty)$, we set for any $(t,\theta)\in [0,1]\times [0,\infty)$,
\begin{equation*}\label{distortion1-1}
\tau_{K,N}^{(t)}(\theta)=
\begin{cases}
t^{\frac{1}{N}}\left(\frac{\sinh\big(t\theta\sqrt{-K/(N-1)}\big)}{\sinh\big(\theta\sqrt{-K/(N-1)}\big)}\right)^{1-\frac{1}{N}},\quad &{\mbox{if }K\theta^2< 0\mbox{ and }N>1},\\
t,\quad &{\hbox{if }K\theta^2=0,\mbox{ or if }K\theta^2<0\mbox{ and }N=1},\\
t^{\frac{1}{N}}\left(\frac{\sin\big(t\theta\sqrt{K/(N-1)}\big)}{\sin\big(\theta\sqrt{K/(N-1)}\big)}\right)^{1-\frac{1}{N}},\quad &{\hbox{if }0<K\theta^2<(N-1)\pi^2},\\
+\infty,\quad &{\hbox{if }K\theta^2\geq (N-1)\pi^2}.
\end{cases}
\end{equation*}
\begin{definition}\label{rcd}
Let $K\in \R$ and $N\in (1,\infty)$. We say that the metric measure space $(X,d,\mu)$ satisfies the Riemannian curvature-dimension condition, denoted as the $\RCD(K,N)$ space, if it is infinitesimally Hilbertian and for every pair $\eta_0,\eta_1\in \mathcal{P}(X,\mu)$ with bounded support, there exists an optimal coupling $\pi$ of $\eta_0$ and $\eta_1$ such that
\begin{eqnarray}\label{cd-1-1}
&&\int_X \rho_t^{1-\frac{1}{N'}}\,\d\mu\cr
 &\geq& \int\left[\tau_{K,N'}^{(1-t)}(d(\gamma_0,\gamma_1))\rho_0^{-\frac{1}{N'}}(\gamma_0) + \tau_{K,N'}^{(t)}(d(\gamma_0,\gamma_1))\rho_1^{-\frac{1}{N'}}(\gamma_1) \right]\,\d\pi(\gamma),
\end{eqnarray}
for all $t\in [0,1]$ and all $N'\geq N$, where, for every $t\in [0,1]$,  $\rho_t$ denotes the Radon--Nikodym derivative $\frac{\d (e_t)_{\#}\pi}{\d\mu}$.
\end{definition}
Note that from Definition \ref{rcd}, we can deduce that, for any $K'\leq K$ and $N'\geq N$, $\RCD(K,N)$ implies $\RCD(K',N)$ and $\RCD(K,N')$.

Recall that we call that the metric measure space $(X,d,\mu)$ is an $\RCD^\ast(K,N)$ space with $K\in\R$ and $N\in (1,\infty)$ if the same conditions in Definition \ref{rcd} are satisfied with $\tau_{K,N'}^{(1-t)}(d(x_0,x_1))$ and $\tau_{K,N'}^{(t)}(d(x_0,x_1))$ in \eqref{cd-1-1} replaced respectively by $\sigma_{K,N'}^{(1-t)}(d(x_0,x_1))$ and $\sigma_{K,N'}^{(t)}(d(x_0,x_1))$, where for any $(t,\theta)\in [0,1]\times [0,\infty)$,
\begin{equation*}\label{distortion2}
\sigma_{K,N}^{(t)}(\theta)=
\begin{cases}
 \frac{\sinh\big(t\theta\sqrt{-K/N}\big)}{\sinh\big(\theta\sqrt{-K/N}\big)} ,\quad &{\mbox{if }K\theta^2< 0\mbox{ and }N>1},\\
t,\quad &{\hbox{if }K\theta^2=0,\mbox{ or if }K\theta^2<0\mbox{ and }N=1},\\
 \frac{\sin\big(t\theta\sqrt{K/N}\big)}{\sin\big(\theta\sqrt{K/N}\big)} ,\quad &{\hbox{if }0<K\theta^2<N\pi^2},\\
+\infty,\quad &{\hbox{if }K\theta^2\geq N\pi^2}.
\end{cases}
\end{equation*}
See \cite[Sections 3 and 4]{eks2013} for other equivalent characterizations of the $\RCD^\ast(K,N)$ space. It turns out that every $\RCD(K,N)$ space is an $\RCD^\ast(K,N)$ space, and every $\RCD^\ast(K,N)$ space is an $\RCD((N-1)K/N,N)$ space. In particular, $\RCD(0,N)$ and $\RCD^\ast(0,N)$ are equivalent.

From now on, let $(X,d,\mu)$ be an $\RCD(K,N)$ space with $K\in\R$ and $N\in(1,\infty)$. Then  the measure $\mu$ satisfies the local doubling (global doubling, provided $K\geq 0$) property, which we present in the next lemma (see e.g.,  \cite{Sturm2006b}, \cite[Section 5]{gi2012} or \cite[Section 2]{JLZ2014}).
\begin{lemma}\label{doubling}
Let $(X,d,\mu)$ be an $\RCD(K,N)$ space with $K\leq 0$ and $N\in (1,\infty)$, and let $x\in X$ and $0<r\leq R<\infty$.
\begin{itemize}
\item[(i)] If $K=0$, then
$$\mu\big(B(x,R)\big)\leq \left(\frac{R}{r}\right)^{N} \mu\big(B(x,r)\big).$$

\item[(ii)] If $K<0$, then
$$\mu\big(B(x,R)\big)\leq \frac{l_{K,N}(R)}{l_{K,N}(r)} \mu\big(B(x,r)\big),$$
where $(0,\infty)\ni t\mapsto l_{K,N}(t)$ is a continuous function depending on $K$ and $N$, and $l_{K,N}(t)=O(e^{tC(K,N)})$ as $t$ tends to $\infty$ for some constant $C(K,N)$ depending on $K$ and $N$.
\end{itemize}
\end{lemma}

Note that, since we only consider the case when $K=0$, we never use Lemma \ref{doubling}(ii) in the main parts of this note, which is presented here just for completeness.

From the definition of the $\RCD(K,N)$ space, we know that $(X,d)$ is a length space. The (local) doubling property immediately implies that every bounded closed ball in $(X,d)$ is totally bounded. Since $(X,d)$ is also complete, it is then proper and geodesic. Recall that a metric space $(X,d)$ is proper if every bounded closed subset is compact. Hence, it is immediate to check that the Dirichlet form $(\mathcal{E},W^{1,2}(X))$, defined at the beginning of this subsection, is strongly local and regular.

By \cite[Theorem 3.9]{AmbrosioGigliSavare2012}, we see that the intrinsic metric induced by the Dirichlet form $(\mathcal{E}, W^{1,2}(X))$, defined as
$$d_{\mathcal{E}}(x,y)=\sup\{\psi(x)-\psi(y):\, \psi\in W^{1,2}(X)\cap C(X),\, |\nabla\psi|_w\leq 1\, \mu\mbox{-a.e. in }X\},$$
for every $x,y\in X$, coincides with the original one, i.e.,
$$d_{\mathcal{E}}(x,y)=d(x,y),\quad \forall\, x,y\in X.$$
Hence, we can work indifferently with either one of the distances $d$ and $d_{\mathcal{E}}$.

Recently, T. Rajala \cite{raj1,raj2} proved that a weak local $L^2$-Poincar\'{e} inequalities hold in the $\RCD(K,N)$ space, and hence also a (strong) local  $L^2$-Poincar\'{e} inequalities hold by the doubling and geodesic properties and by applying \cite[Theorem 1]{HajlaszKoskela1995}. See also \cite[section 2]{JLZ2014}.

\begin{lemma}\label{localPoicare}
Let $(X,d,\mu)$ be an $\RCD(K,N)$ space with $K\leq 0$ and $N\in (1,\infty)$. Then for every $x\in X$ and every $R>0$, there exists a positive constant $C:=C(K,N,R)$ such that for any $r\in (0,R)$,
\begin{eqnarray}\label{localstrongPoincare}
\int_{B(x,r)} |f-f_B|^2\, \d\mu\leq Cr^2\int_{B(x,r)}|\nabla f|^2\, \d\mu,\quad\mbox{for all }f\in W^{1,2}(X),
\end{eqnarray}
where $f_B=\frac{1}{\mu(B(x,r))}\int_{B(x,r)} f\, \d\mu$. In particular, if $K=0$, then \eqref{localstrongPoincare} holds with constant $C:=C(N)$ independent of $R$.
\end{lemma}

Now we can apply the results obtained by Sturm in \cite[Proposition 2.3]{st2} to immediately deduce that there exist a heat kernel, i.e., a measurable map $(0,\infty)\times X\times X \ni (t,x,y)\mapsto p_t(x,y)\in [0,\infty)$ such that, for any $ t>0$, $f\in L^1(X)+L^\infty(X)$ and $\mu$-a.e. $x\in X$,
$$P_tf(x)=\int_X f(y)p_t(x,y)\,\d\mu(y);$$
for all $s,t>0$ and $\mu$-a.e. $x,y\in X$,
$$p_{t+s}(x,y)=\int_X p_t(x,z)p_s(z,y)\,\d\mu(z);$$
the function $(t,y)\mapsto p_t(x,y)$ is a solution of the equation $\Delta u=\frac{\partial}{\partial t}u$ on $(0,\infty)\times X$ in the distribution sense (see also Definition \ref{weaksolution} below). By the symmetry of the semi-group, $p_t$ is also symmetric, i.e., for every $t>0$, $p_t(x,y)=p_t(y,x)$ for $\mu\times\mu$-a.e. $(x,y)\in X\times X$. The doubling property and the local $L^2$-Poincar\'{e} inequality imply that the function $x\mapsto p_t(x,y)$ is H\"{o}lder continuous for every $(t,y)\in (0,\infty)\times X$, by a standard iteration argument; see e.g. \cite[Section 3]{st3}. Moreover, $P_t$ is stochastically complete (see e.g. \cite[Theorem 4]{stm1}), i.e.,
\begin{eqnarray}\label{stocom}
\int_Xp_t(x,y)\, \d\mu(y)=1,\quad\mbox{for any }t>0\mbox{ and }x\in X,
\end{eqnarray}
and, when $K=0$, the following upper and lower estimates of Gaussian type hold:
\begin{eqnarray}\label{fullgaussian}
\frac{C(N)^{-1}}{\mu(B(x,\sqrt t))}\exp\left\{-C(N)\frac{d^2(x,y)}{t}\right\}&\leq& p_t(x,y)\cr
&\leq&\frac{C_0(N)}{\mu(B(x,\sqrt t))}\exp\left\{-\frac{d^2(x,y)}{5t}\right\},
\end{eqnarray}
where $C(N)$ and $C_0(N)$ are positive constants depending only on $N$.

For any $K\in\R$ and $N\in(1,\infty)$, define the function $\tilde{\tau}_{K,N}: [0,\infty)\rightarrow\R$ by
\begin{equation*}
\tilde{\tau}_{K,N}(s)=
\begin{cases}
\frac{1}{N}\left[1+s\sqrt{K(N-1)}\cot\left(s\sqrt{\frac{K}{N-1}}\right)\right],\quad &{\hbox{if}}\,\ K>0,\\
1,\quad &{\hbox{if}}\,\ K=0,\\
\frac{1}{N}\left[1+s\sqrt{-K(N-1)}\coth\left(s\sqrt{\frac{-K}{N-1}}\right)\right],\quad &{\hbox{if}}\,\ K<0.
\end{cases}
\end{equation*}

N. Gigli proved the following Laplacian comparison principle in \cite{gi2012}. Here and in what follows, for $x\in X$, let $d_x=d(x,\cdot): X\rightarrow [0,\infty)$ be the distance function from the fix point $x$.
\begin{lemma}[Laplacian comparison principle]\label{lap-comp}
Let $(X,d,\mu)$ be an $\RCD(K,N)$ space with $K\in\R$ and $N\in (1,\infty)$.
Then, for every $o\in X$, $d_{o}\in \mathcal D_\loc(\Delta,\,X\setminus{o})$ and
$$\Delta d_{o}|_{X\setminus{o}}\leq \frac{N\tilde{\tau}_{K,N}(d_{o})-1}{d_{o}}.$$
\end{lemma}

Finally, in what follows, we let $\mathbb{N}=\{1,2,\cdots\}$ and $B(x,r)$ be the ball in $(X,d)$ with center $x$ and radius $r>0$.

\section{Sharp heat kernel bounds}
\hskip\parindent The following lemmata are important in the establishment of the sharp heat kernel bounds presented below. Let $\Omega$ be an open subset of $(X,d)$.
\begin{definition}\label{weaksolution}
Let $I$ be an open interval in $\R$, and $g\in L^2(\Omega)$. We call that a function $u: I\rightarrow W^{1,2}(\Omega)$ satisfies the parabolic equation
$$\frac{\partial}{\partial t}u-\Delta u\leq g,\quad\mbox{in }I\times \Omega,$$
if for every $t\in I$, the Fr\'{e}chet derivative of $u$, denoted by $\frac{\partial}{\partial t}u$, exists in $L^2(\Omega)$  and for any nonnegative function $\psi\in W^{1,2}(\Omega)$, it holds
$$\int_\Omega \frac{\partial}{\partial t}u(t,\cdot)\psi\,\d\mu + \mathcal{E}\big(u(t,\cdot),\psi\big)\leq \int_\Omega g\psi\,\d\mu.$$
\end{definition}
In a similar way, one can define the solution to the parabolic equations $\frac{\partial}{\partial t}u-\Delta u\geq g$ and $\frac{\partial}{\partial t}u-\Delta u=g$ in $I\times \Omega$.

The first lemma is on the parabolic maximum principle for the heat equation. The proof is essentially from \cite[Section 4.1]{GrigoryanHu2008} and can be simplified a little bit in our context. So we omit the proof here. We shall point out that the metric measure space $(X,d,\mu)$ is also locally compact under the $\RCD(K,N)$ condition with $K\in\R$ and $N\in(1,\infty)$ (see \cite[Corollary 2.4]{Sturm2006b}).
\begin{lemma}[Parabolic maximum principle]\label{para-max}
Let $(X,d,\mu)$ be an $\RCD(K,N)$ space with $K\in \R$ and $N\in (1,\infty)$. Fix $T\in (0,\infty]$. Assume that a function $u: (0,T)\rightarrow W^{1,2}(\Omega)$, with $u_+(t,\cdot)=\max\{u(t,\cdot),0\}\in W^{1,2}(\Omega)$ for any $t\in (0,T)$, satisfies the following equation with initial value condition:
\begin{equation*}
\begin{cases}
\frac{\partial}{\partial t}u-\Delta u\leq 0,\quad &{\hbox{in}}\,\ (0,T)\times\Omega,\\
u_+(t,\cdot)\rightarrow 0,\quad &{\hbox{in}}\,\ L^2(\Omega)\,\ \hbox{as}\,\ t\rightarrow0.
\end{cases}
\end{equation*}
Then $u(t,x)\leq0$ for any $t$ in $(0,T)$ and $\mu$-a.e. $x$ in $\Omega$.
\end{lemma}

As an application of the Laplacian comparison principle in Lemma \ref{lap-comp} and the parabolic maximum principle in Lemma \ref{para-max}, we derive the following heat kernel comparison results, which generalize the results obtained by Cheeger--Yau \cite{ChYau1981} and Li--Yau \cite{LiYau}. We should mention that the proof is more or less standard, which we present here for the sake of completeness. Let $B(p,r)$ denote the ball in $X$ with center $p$ and radius $r$ with respect to the metric $d$, and let $\mathbb{M}^{K,N}$ be the complete and simply connected space form with sectional curvature $K\in\R$ and dimension $N\in\mathbb{N}$. For any $t>0$, denote by $p_t: \Omega\times\Omega\rightarrow\R$ the Dirichlet heat kernel on $\Omega$, and by $\bar{p}_t: \bar{B}(\bar{x},r)\times \bar{B}(\bar{x},r)\rightarrow \R$ the Dirichlet heat kernel on $\bar{B}(\bar{x},r)$, which is a geodesic ball with center $\bar{x}\in\mathbb{M}^{K,N}$ and radius $r>0$ with respect to the distance $\bar{d}$ in $\mathbb{M}^{K,N}$. For $\bar{z}\in \mathbb{M}^{K,N}$, let $\d\bar{z}$ be the volume measure in $\mathbb{M}^{K,N}$.
\begin{lemma}[Heat kernel comparison]\label{kernelcomparison}
Let $(X,d,\mu)$ be an $\RCD((N-1)K,N)$ space with $K\in \R$ and $N\in \mathbb{N}\setminus\{1\}$, and let $B(x,r)$ be a ball with center $x\in X$ and radius $r>0$ with respect to $d$. Suppose the function $h_0$ belongs to $C^1([0,r], [0,\infty))$ and satisfies that $h_0'(0)=0$, $h_0(r)=0$ and $h'(s)\leq0$ for any $s\in (0,r]$. Let $\bar{B}(\bar{x},r)$ be a ball with center $\bar{x}\in M^{K,N}$.  Then, for any $t>0$ and $y\in B(x,r)$ with $\bar{y}\in \bar{B}(\bar{x},r)$ such that $d(x,y)=\bar{d}(\bar{x},\bar{y})$,
$$\int_X p_t(z,y)h_0(d(x,z))\,\d\mu(z)\geq \int_{\bar{B}(\bar{x},r)}\bar{p}_t(\bar{z},\bar{y})h_0(\bar{d}(\bar{x},\bar{z}))\,\d\bar{z};$$
in particular,
$$\int_{B(x,\lambda)}p_t(z,y)\,\d\mu(z)\geq \int_{\bar{B}(\bar{x},\lambda)}\bar{p}_t(\bar{z},\bar{y})\,\d\bar{z},\quad\mbox{for any }\lambda\in [0,r],$$
and
$$p_t(x,y)\geq \bar{p}_t(\bar{x},\bar{y}).$$
\end{lemma}
\begin{proof}
Let $h_0(\cdot)=h_0\circ d_x$ be a nonnegative function of the distance $d$ to the fixed point $x\in X$, and let
$$h(t,x)=\int_X p_t(x,z)h_0(z)\,\d\mu(z)$$
be the solution to the heat equation in $X$ with initial data $h_0$, i.e.,
\begin{equation*}
\begin{cases}
\frac{\partial}{\partial t}u-\Delta u=0,\quad &{\hbox{in}}\,\ (0,T)\times X,\\
u_+(t,\cdot)\rightarrow h_0,\quad &{\hbox{in}}\,\ L^2(X)\,\ \hbox{as}\,\ t\rightarrow0.
\end{cases}
\end{equation*}
Let
$$\bar{h}(t,\bar{y})=\int_{\bar{B}(\bar{x},r)}\bar{p}_t(\bar{y},\bar{z})h_0(\bar{d}(\bar{x},\bar{z}))\,\d\bar{z}$$
be the solution to the heat equation on $\bar{B}(\bar{x},r)$ with the Dirichlet boundary condition. It is easy to know that the rotational symmetry of $h_0$ implies that the function $\bar{h}(t,\cdot)$ is also rotationally symmetric for every $t>0$. Hence, we can write $\bar{h}(t,\bar{y})=\bar{h}(t,\bar{d}(\bar{x},\bar{y}))$. We claim that
\begin{eqnarray}\label{barlap-0}\bar{h}'(t,\bar{r}):=\frac{\partial}{\partial \bar{r}}\bar{h}(t,\bar{r})\leq 0,\quad\mbox{for all }t\geq 0\mbox{ and }\bar{r}\in [0,r].
\end{eqnarray}
Indeed, letting $\bar{\Delta}$ be the Laplacian on $\mathbb{M}^{K,N}$, we have that
\begin{eqnarray}\label{barlap-1}
0=\left(\frac{\partial}{\partial t}-\bar{\Delta}\right)\bar{h}(t,\bar{r})=\frac{\partial}{\partial t}\bar{h}-\bar{h}'\bar{\Delta}\bar{d}_{\bar{x}} -\bar{h}''
\end{eqnarray}
with
\begin{equation*}
\bar{\Delta}\bar{d}_{\bar{x}}=
\begin{cases}
(N-1)\sqrt{K}\cot\left(\sqrt{K} \bar{d}_{\bar{x}}\right),\quad &{\hbox{if}}\,\ K>0,\\
(N-1)\bar{d}_{\bar{x}}^{-1},\quad &{\hbox{if}}\,\ K=0,\\
(N-1)\sqrt{-K}\coth\left(\sqrt{-K}\bar{d}_{\bar{x}}\right),\quad &{\hbox{if}}\,\ K<0,
\end{cases}
\end{equation*}
in the distribution sense. By direct differentiation with respect to $\bar{d}_{\bar{x}}$, we derive from \eqref{barlap-1} that
$$0=\frac{\partial}{\partial t}v-v'\bar{\Delta}\bar{d}_{\bar{x}}-v(\bar{\Delta}\bar{d}_{\bar{x}})'-v''$$
with $v=\bar{h}'$ and
\begin{equation*}
(\bar{\Delta}\bar{d}_{\bar{x}})'=
\begin{cases}
-(N-1)K\csc^2\left(\sqrt{K} \bar{d}_{\bar{x}}\right),\quad &{\hbox{if}}\,\ K>0,\\
-(N-1)\bar{d}_{\bar{x}}^{-2},\quad &{\hbox{if}}\,\ K=0,\\
(N-1)K\, \csch^2 \left(\sqrt{-K}\bar{d}_{\bar{x}}\right),\quad &{\hbox{if}}\,\ K<0,
\end{cases}
\end{equation*}
which implies that $(\bar{\Delta}\bar{d}_{\bar{x}})'\leq 0$. Thus, by the assumption on $h_0$, we prove the claim \eqref{barlap-0}.

Now let $\bar{h}(t,y)=\bar{h}(t,d(x,y))$. Combining the Laplacian  comparison principle in Lemma \ref{lap-comp}, \eqref{barlap-0} and \eqref{barlap-1}, we obtain that
\begin{eqnarray*}
\left(\frac{\partial}{\partial t}-\Delta\right)\bar{h}(t,d(x,\cdot))&=&\frac{\partial}{\partial t}\bar{h}-\bar{h}'\Delta d_x -\bar{h}''\\
&\leq&\frac{\partial}{\partial t}\bar{h}-\bar{h}'\Delta \bar{d}_{\bar{x}} -\bar{h}''=0
\end{eqnarray*}
holds in the distribution sense in $X$, and $\bar{h}(0,y)=h_0(y)$ for any $y\in X$. Denote $G=\bar{h}-h$. Then $G$ satisfies the equation
\begin{equation*}
\begin{cases}
\frac{\partial}{\partial t}u-\Delta u\leq0,\quad &{\hbox{in}}\,\ (0,\infty)\times B(x,r),\\
u_+(t,\cdot)\rightarrow 0,\quad &{\hbox{in}}\,\ L^2(B(x,r))\,\ \hbox{as}\,\ t\rightarrow0.
\end{cases}
\end{equation*}
Thus, the parabolic maximum principle in Lemma \ref{para-max} implies that
$$\bar{h}(t,d(x,y))\leq h(t,y),$$
for any $t>0$ and $\mu$-a.e. $y$ in $B(x,r)$.

In particular, for the second and last assertions, we need to approximate the characteristic function of $B(x,\lambda)$ and the dirac function at $x$ with a sequence of functions satisfy the requirement of $h_0$, respectively.
\end{proof}

The following result is borrowed from \cite{jia14} (see also \cite{GarofaloMondino} for the case when the reference measure $\mu$ is a Borel probability measure), and we present it here for convenience.
\begin{lemma}[Parabolic Harnack inequality]\label{para-harnack}
Let $(X,d,\mu)$ be an $\RCD(0,N)$ space with $N\in (1,\infty)$.
Given any $x,y,z\in M$ and $0<s<t<\infty$, it holds that
$$p_s(x,y)\leq p_t(x,z)\exp\left\{\frac{d(y,z)^2}{4(t-s)}\right\}\left(\frac t s\right)^{N/2}.$$
\end{lemma}

Now we recall the definition of the Minkowski content which will be used in the proof of Theorem \ref{main} below.
\begin{definition}[Minkowski content]
Let $x\in X$ and $r\in [0,\infty)$. Define the  Minkowski content of the ball $B(x,r)$ by
$$s(x,r):=\limsup_{\delta\rightarrow 0}\frac 1\delta \mu\big(B(x ,r+\delta)\setminus B(x ,r)\big).$$
\end{definition}

We remark here that, in the Riemannian manifold $(M,d,\mu)$ with $\mu$ the Riemannian volume measure and $d$ the Riemannian distance, it is immediate to see that, for every geodesic ball $B(x,r)$ in $(M,d)$,  $s(x ,r)$ is equal to the $(n-1)$-dimensional Hausdorff measure of $\partial B(x ,r)$ (see e.g. \cite{Chavel1993}).

The first part in the next lemma is known (see \cite[Theorem 2.3]{Sturm2006b}), and the second part is immediate from the last definition and the local Lipschitz continuity of the function $r\mapsto\mu(B(x,r))$ in $(0,\infty)$, for each $x\in X$ (see e.g. \cite[p.148]{Sturm2006b}).
\begin{lemma}\label{volume-growth}
Let $(X,d,\mu)$ be a $RCD(0,N)$ spaces with $N\in (1,\infty)$. Then for all $x\in X$ and $0<r\leq R<\infty$, it holds
$$\frac{s(x,R)}{s(x,r)}\le \left(\frac Rr\right)^{N-1},$$
and
$$\mu(B(x,R))=\int_0^Rs(x,t)\,\d t.$$
\end{lemma}

Now we recall the definition of maximum volume growth.
\begin{definition}[Maximum volume growth]\label{def-of-maxvol}
Let $(X,d,\mu)$ be an $\RCD(0,N)$ space with $N\in (1,\infty)$. It is said to have maximum volume growth if, for some point $x\in X$, there exists a constant $\kappa_\infty>0$ such that
\begin{eqnarray}\label{max-vol-1}
\liminf_{r\rightarrow\infty}\frac{\mu(B(x,r))}{r^{N}}=\kappa_\infty.
\end{eqnarray}
\end{definition}
\begin{remark} It is easy to show that the limit $\liminf_{r\rightarrow\infty}[\mu(B(x,r))/r^N]$ is independent of $x$; hence it can be considered as a global geometric invariant of $(X,d,\mu)$.
\end{remark}

\begin{lemma}\label{lowerboundary}
Let $(X,d,\mu)$ be a $RCD(0,N)$ spaces with $N\in (1,\infty)$ having maximum volume growth. Then
$$\liminf_{r\rightarrow\infty}\frac{s(x,r)}{Nr^{N-1}}=\kappa_\infty;$$
moreover, for any $r>0$, it holds that
$$s(x,r)\ge N\kappa_\infty r^{N-1}.$$
\end{lemma}
\begin{proof}
On the one hand, by the doubling property in Lemma \ref{doubling}(i),
\begin{eqnarray*}
s(x,r)&=&\limsup_{\delta\rightarrow 0}\frac 1\delta \mu\big(B(x,r+\delta)\setminus B(x,r)\big)\\
&\leq&\limsup_{\delta\rightarrow 0}\frac 1\delta \left[\Big(\frac{r+\delta}{r}\Big)^N\mu(B(x,r))-\mu(B(x,r))\right]\\
&=&N\frac{\mu(B(x,r))}{r}.
\end{eqnarray*}
Thus, by the maximum volume growth, we derive that for any $\epsilon>0$, there exists $A>0$ such that for any $r\geq A$, it holds that
$\mu(B(x,r))\leq (1+\epsilon)\kappa_\infty r^N$, and hence
$$s(x,r)\leq (1+\epsilon)\kappa_\infty N r^{N-1}.$$

On the other hand, it follows from Lemma \ref{volume-growth} that the function $(0,\infty)\ni r\mapsto\frac{s(x,r)}{r^{N-1}}$ is non-increasing, which immediately implies that if there exists $r_0\in (0,\infty)$ such that
$s(x,r_0)< N\kappa_\infty r_0^{N-1},$ then, for all $r>r_0$, we have that
$s(x,r)< N\kappa_\infty r^{N-1}.$
Applying Lemma \ref{volume-growth} again, we derive that for each $r>r_0$,
$$\mu(B(x,r)\setminus B(x,r_0))=\int_{r_0}^r s(x,t)\,\d t< \kappa_\infty\big(r^N-r_0^{N}\big),$$
and hence
$$\liminf_{r\rightarrow\infty}\frac{\mu(B(x,r))}{r^N}<\liminf_{r\rightarrow\infty}\frac{\kappa_\infty \left(r^N-r_0^{N}\right)+\mu(B(x,r_0))}{r^N}=\kappa_\infty,$$
which is a contradiction. Thus, for any $r>0$, we get that $s(x,r)\ge N\kappa_\infty r^{N-1}$, which is the last assertion.

Combing the above results, we finally reach the first assertion.
\end{proof}

\begin{definition}[Boundary Integral]
Let $o\in X$ and $r\in [0,\infty)$. Suppose $f\in L^\infty_\loc(X)$. Define the integral of $f$ on $\partial B(o,r)$ as
$$|f|_{\partial B(o,r)}=\limsup_{\delta\rightarrow 0^{+}}\frac 1\delta \int_{B(o,r+\delta)\setminus B(o,r)}f(z)\,\d\mu(z).$$
\end{definition}

The next result is a substitute for the co-area formula in Riemannian manifolds in our more general setting to some extent, which is important for the proof of Theorem \ref{main} below.
\begin{lemma}\label{co-area}
Let $(X,d,\mu)$ be an $\RCD(K,N)$ space with $K\in \R$ and $N\in (1,\infty)$. Let $o\in X$ and $R\in (0,\infty)$.
Then for each $f\in L^\infty_\loc(X)$, it holds that
\begin{eqnarray}\label{co-area-1}
\int_{B(o,R)}f\,\d\mu=\int_0^R |f|_{\partial B(o,r)}\,\d r.
\end{eqnarray}
In addition, if $\phi: \R_+\rightarrow\R_+$ is a locally continuous and monotone function, then
\begin{eqnarray}\label{co-area-2}
\int_{B(o,R)}\phi(d(o,z))f(z)\,\d\mu(z)=\int_0^R \phi(r)|f|_{\partial B(o,r)}\,\d r,
\end{eqnarray}
and
\begin{eqnarray}\label{co-area-3}
\int_{X\setminus B(o,R)}\phi(d(o,z))f(z)\,\d\mu(z)=\int_R^\infty \phi(r)|f|_{\partial B(o,r)}\,\d r.
\end{eqnarray}
\end{lemma}
\begin{proof}
Notice that, for each fixed $o\in X$, the function $r\mapsto\mu(B(o,r))$ is locally Lipschitz continuous on $(0,\infty)$.
From this,  we conclude that, for each $f\in L^\infty_\loc(X)$,
the function
$$r\mapsto\int_{B(o,r)}f\,d\mu$$
is locally Lipschitz continuous on $(0,\infty)$, and hence, the required equality \eqref{co-area-1} holds.

Without loss of generality, assume that the function $\phi: \R_+\rightarrow\R_+$ is locally continuous and monotonically decreasing. Let $r>0$. For any $\delta>0$, note that for any $z\in B(o,r+\delta)\setminus B(o,r)$, $r\leq d(o,z)<r+\delta$. Then, if $f$ is nonnegative, then
\begin{eqnarray*}
\int_{B(o,r+\delta)\setminus B(o,r)}\phi(r+\delta) f(z)\,\d\mu(z)&\leq&\int_{B(o,r+\delta)\setminus B(o,r)}\phi(d(o,z)) f(z)\,\d\mu(z)\\
&\leq& \int_{B(o,r+\delta)\setminus B(o,r)}\phi(r) f(z)\,\d\mu(z),
\end{eqnarray*}
and, if $f$ is non-positive, then
\begin{eqnarray*}
\int_{B(o,r+\delta)\setminus B(o,r)}\phi(r+\delta) f(z)\,\d\mu(z)&\geq&\int_{B(o,r+\delta)\setminus B(o,r)}\phi(d(o,z)) f(z)\,\d\mu(z)\\
&\geq& \int_{B(o,r+\delta)\setminus B(o,r)}\phi(r) f(z)\,\d\mu(z);
\end{eqnarray*}
hence, by definition,
$$|(\phi\circ d_o) f|_{\partial B(o,r)}=\phi(r)|f|_{\partial B(o,r)}.$$
Thus,
\begin{eqnarray*}
\int_{B(o,R)}\phi(d(o,z))f(z)\,d\mu(z)=\int_0^R |(\phi\circ d_o) f|_{\partial B(o,r)}\,dr=\int_0^R \phi(r) | f|_{\partial B(o,r)}\,dr,
\end{eqnarray*}
which completes the proof of \eqref{co-area-2}. The proof of \eqref{co-area-3} is similar.
\end{proof}

Now we present the main result in the next theorem. Fix a point $x\in X$. For $N\in \mathbb{N}$, let $\omega_N$ be the volume of the unit ball in $\R^N$. For any $r>0$, set
\begin{equation}\label{kappa-x}
\kappa_x(r)=\frac{\mu(B(x,r))}{r^N}.
\end{equation}
Then from Lemma \ref{doubling}, it is immediate to know that, in the $\RCD(0,N)$ space with $N\in (1,\infty)$, the function $r\mapsto\kappa_x(r)$ is locally Lipschitz continuous in $(0,\infty)$, and monotonically decreases to $\kappa_\infty$ as $r$ increases to $\infty$.

Now we present the main result in this section in the next theorem.
\begin{theorem}\label{main}
Let $(X,d,\mu)$ be an $\RCD(0,N)$ space with $N\in \mathbb{N}\setminus\{1\}$ having the maximum volume growth \eqref{max-vol-1}.
Then, for any $\epsilon>0$ and all $y\in X$, it holds that
\begin{eqnarray}\label{UE}
p_t(x,y)\leq (1+C_N(\epsilon+\beta))\frac{\omega_N}{\kappa_\infty}(4\pi t)^{-\frac{N}{2}}\exp\left\{-\frac{1-\epsilon}{4t}d(x,y)^2 \right\},
\end{eqnarray}
and
\begin{eqnarray}\label{LE}
p_t(x,y)\geq\frac{\omega_N}{\kappa_x(\epsilon d(x,y))}(4\pi t)^{-\frac{N}{2}}\exp\left\{ -\frac{1+\epsilon(\epsilon+2)^2}{4t}d(x,y)^2\right\},
\end{eqnarray}
where $C_N$ is a positive constant depending only on $N$, $\kappa_x(\cdot)$ is defined in \eqref{kappa-x}, and $\beta$ is given by
\begin{eqnarray}\label{beta}
\beta=\epsilon^{-2N}\max_{r\geq (1-\epsilon)d(x,y)}\left[1-\frac{\kappa_\infty}{\kappa_x(\epsilon^{2N+1}r)}\right].
\end{eqnarray}
\end{theorem}
\begin{remark}
We should mention that the idea of proof of the theorem is from \cite{LiTamWang}. Note that, due to our calculation, on the one hand, the constant $C_N$ in \eqref{UE} above depends only on $N$, while the constant in the expression of the heat kernel upper estimate in \cite[Theorem 2.1]{LiTamWang} depends not only on $N$ but also on $\kappa_\infty$, and on the other hand, from Lemma 3.6, we have
\begin{eqnarray*}
\kappa_x(r)=\frac{1}{r^N}\int_0^r s(x,t)\,\d t\geq\frac{1}{r^N}\int_0^r s(x,r)\Big(\frac{t}{r}\Big)^{N-1}\,\d t=\frac{s(x,r)}{Nr^{N-1}},
\end{eqnarray*}
or $s(x,r)\leq Nr^{N-1}\kappa_x(r)$, but what we need in the proof of Theorem \ref{main} is a lower bound on $s(x,r)$ (see Lemma \ref{lowerboundary} above). For the latter, the cost is that $\kappa_\infty$ goes into the definition of $\beta$ (see \eqref{beta} above). So we should present the detailed proof here.
\end{remark}
\begin{proof}[Proof of Theorem \ref{main}]
(1) We first prove the lower bound of the heat kernel. Let $y\in X$. Applying the parabolic Harnack inequality in Lemma \ref{para-harnack}, we deduce that for any $\delta>0$ and any $z\in X$,
\begin{eqnarray}\label{harnack}
p_t(z,y)\leq (1+\delta)^{\frac{N}{2}}\exp\left\{\frac{d(z,x)^2}{4\delta t}\right\}p_{(1+\delta)t}(x,y).
\end{eqnarray}
Combining with the heat kernel comparison result in Lemma \ref{kernelcomparison}, we have for any $R>0$,
\begin{eqnarray*}
&&\frac{1}{\mu(B(x,R))}\int_{B(\bar{x},R)}\bar{p}_{(1+\delta)t}(\bar{z},\bar{y})\,\d\bar{z}\leq\frac{1}{\mu(B(x,R))}
\int_{B(x,R)}p_{(1+\delta)t}(z,y)\,\d\mu(z)\\
&\leq&\frac{1}{\mu(B(x,R))}\int_{B(x,R)}p_{(1+\delta)^2t}(x,y)(1+\delta)^{\frac{N}{2}}\exp\left\{\frac{d(z,x)^2}{4\delta(1+\delta)t}\right\}\,\d \mu(z)\\
&\leq&p_{(1+\delta)^2t}(x,y)(1+\delta)^{\frac{N}{2}}\exp\left\{ \frac{R^2}{4\delta(1+\delta)t}\right\},
\end{eqnarray*}
which is
\begin{eqnarray*}\label{comparison}
p_{(1+\delta)^2t}(x,y)&\geq&(1+\delta)^{-\frac{N}{2}}\mu(B(x,R))^{-1}\exp\left\{-\frac{R^2}{4\delta(1+\delta)t}\right\}\int_{B(\bar{x},R)}
\bar{p}_{(1+\delta)t}(\bar{z},\bar{y})\,\d \bar{z}\\
&\geq& \omega_N R^N\mu(B(x,R))^{-1}(1+\delta)^{-N}\exp\left\{-\frac{(\delta+2)R^2}{4\delta(1+\delta)t}\right\}\bar{p}_t(\bar{x},\bar{y}),
\end{eqnarray*}
where $d(z,y)=d(\bar{z},\bar{y})$ and we used the parabolic Harnack inequality in the last line. Since $\mu(B(x,R))=\kappa_x(R)R^N$, setting $R=\delta d(x,y)$, we have
$$p_{(1+\delta)^2t}(x,y)\geq\frac{\omega_N}{\kappa_x(\delta d(x,y))}(1+\delta)^{-N}\exp\left\{-\frac{\delta(\delta+2)d(x,y)^2}{4(1+\delta)t}\right\}\bar{p}_t(\bar{x},\bar{y}).$$
Letting $s=(1+\delta)^2t$, we obtain that for any $\delta>0$,
\begin{eqnarray}\label{LowerE}
p_s(x,y)&\geq& \frac{\omega_N}{\kappa_x(\delta d(x,y))}(4\pi s)^{-\frac{N}{2}}\exp\left\{-\frac{1+\delta(\delta+2)^2}{4s}d(x,y)^2\right\}.
%&\geq&\frac{\omega_N}{\kappa_p(\delta d(x,y))}(4\pi s)^{-\frac{N}{2}}\exp\left\{-\frac{9\delta+1}{4s}d(x,y)^2\right\}, where the last inequality holds, provided $\delta\in (0,1]$}
\end{eqnarray}
which is the desired lower bound \eqref{LE}.

(2) Now we turn to prove the upper bound. Note that, by the parabolic Harnack inequality in Lemma \ref{para-harnack}, we have that
$$s^{\frac{N}{2}}p_s(y,x)\leq t^{\frac{N}{2}}p_t(z,x)\exp\left\{\frac{d(z,y)^2}{4(t-s)}\right\},$$
for all $0<s<t$. Letting $z=y=x$, it is immediately to observe that the function $t\mapsto t^{\frac{N}{2}}p_t(x,x)$ is monotone nondecreasing such that (see \cite[Theorem 4.1]{JLZ2014})
\begin{eqnarray}\label{main-1}
\lim_{t\rightarrow\infty}t^{\frac{N}{2}}p_t(x,x)=\frac{\omega_N}{\kappa_\infty}(4\pi)^{-\frac{N}{2}}.
\end{eqnarray}
Let $\epsilon$ be sufficiently small constant in $(0,1)$. Suppose $d(x,y)\leq \epsilon\sqrt{t}$. Then, by \eqref{harnack} and \eqref{main-1}, we have that
\begin{eqnarray}\label{near}
p_t(x,y)&\leq& t^{-\frac{N}{2}}[(1+\epsilon)t]^{\frac{N}{2}}p_{(1+\epsilon)t}(x,x)\exp\left\{\frac{d(x,y)^2}{4\epsilon t}\right\}\cr
&\leq& (4\pi t)^{-\frac{N}{2}}\frac{\omega_N}{\kappa_\infty}\exp\left\{\frac{\epsilon}{4}\right\}\cr
&\leq& (4\pi t)^{-\frac{N}{2}}\frac{\omega_N}{\kappa_\infty}\exp\left\{\frac{\epsilon}{4}\right\}\exp\left\{-\frac{d(x,y)^2}{4t}+\frac{\epsilon^2}{4}\right\}\cr
&\leq& (1+C_1\epsilon)(4\pi t)^{-\frac{N}{2}}\frac{\omega_N}{\kappa_\infty}\exp\left\{-\frac{d(x,y)^2}{4t}\right\},
\end{eqnarray}
for some constant $C_1>0$.

Suppose $d(x,y)>\epsilon\sqrt{t}$ now. Set $R=(1-\epsilon)d(x,y)$. With the fact that (see \eqref{stocom})
$$\int_X p_t(x,y)\,\d\mu(y)=\int_{\R^N}\bar{p}_t(\bar{x},\bar{y})\,\d\bar{y}=1,\quad\mbox{for every }t>0,$$
we apply \eqref{comparison} to deduce that
$$\int_{X\setminus B(x,R)}p_{(1+\epsilon)t}(x,y)\,\d\mu(y)\leq \int_{\R^N\setminus B(\bar{x},R)}\bar{p}_{(1+\epsilon)t}(\bar{x},\bar{y})\,\d\bar{y}.$$
Since $B(x,R)\cap B(y,\epsilon R)=\emptyset$, we have
$$B(y,\epsilon R)=\big(X\setminus B(x,R)\big)\setminus\left[X\setminus \big(B(x,R)\cup B(y,\epsilon R)\big)\right],$$
and hence,
\begin{eqnarray*}
&&\int_{B(y,\epsilon R)} p_{(1+\epsilon)t}(x,z)\,\d\mu(z)\\
&=&\int_{X\setminus B(x,R)}p_{(1+\epsilon)t}(x,z)\,\d\mu(z)-\int_{X\setminus \big(B(x,R)\cup B(y,\epsilon R)\big)}p_{(1+\epsilon)t}(x,z)\,\d\mu(z)\\
&\leq& \int_{\R^N\setminus B(\bar{x},R)}\bar{p}_{(1+\epsilon)t}(\bar{x},\bar{z})\,\d\bar{z}- \int_{X\setminus \big(B(x,R)\cup B(y,\epsilon R)\big)} p_{(1+\epsilon)t}(x,z)\,\d\mu(z)\\
&\leq&\int_{\R^N\setminus B(\bar{x},R)}\bar{p}_{(1+\epsilon)t}(\bar{x},\bar{z})\,\d\bar{z}+\int_{B(y,\epsilon R)} p_{(1+\epsilon)t}(x,z)\,\d\mu(z)\\
&&-\int_{X\setminus B(x,R)} p_{(1+\epsilon)t}(x,z)\,\d\mu(z).
\end{eqnarray*}
Since $\kappa_x(r)$ decreases monotonically to $\kappa_\infty$ as $r$ increases to $\infty$, by the lower bound \eqref{LowerE},  \eqref{co-area-2} and \eqref{co-area-3}, we derive that
\begin{eqnarray*}
&&\int_{B(y,\epsilon R)}p_{(1+\epsilon)t}(x,z)\,\d\mu(z)\\
&\leq& \int_{\R^N\setminus B(\bar{x},R)}(4\pi(1+\epsilon)t)^{-\frac{N}{2}}\exp\left\{-\frac{|\bar{x}-\bar{z}|^2}{4(1+\epsilon)t}\right\}\,\d\bar{z}\\
&&- (4\pi(1+\epsilon))^{-\frac{N}{2}}\int_{X\setminus B(x,R)} \frac{\omega_N}{\kappa_x(\epsilon d(x,z))}\exp\left\{-\frac{(1+\delta(\delta+2)^2)d(x,z)^2}{4(1+\epsilon)t}\right\}\,\d\mu(z)\\
&&+ (4\pi(1+\epsilon))^{-\frac{N}{2}}\int_{B(y,\epsilon R)} \frac{\omega_N}{\kappa_x(\epsilon d(x,z))}\exp\left\{-\frac{(1+\delta(\delta+2)^2)d(x,z)^2}{4(1+\epsilon)t}\right\}\,\d\mu(z)\\
&\leq&  N\omega_N (4\pi(1+\epsilon))^{-\frac{N}{2}}\int_R^\infty  r^{N-1}\exp\left\{-\frac{r^2}{4(1+\epsilon)t}\right\}\,\d r\\
&&- N\omega_N(4\pi(1+\epsilon))^{-\frac{N}{2}}\int_R^\infty \frac{s(x,r)}{N\kappa_x(\epsilon r)}\exp\left\{-\frac{(1+\delta(\delta+2)^2)r^2}{4(1+\epsilon)t}\right\}\,\d r\\
&&+ (4\pi(1+\epsilon))^{-\frac{N}{2}}\frac{\omega_N}{\kappa_\infty}\mu(B(y,\epsilon R))\exp\left\{-\frac{(1+\delta(\delta+2)^2)}{4(1+\epsilon)t}\big(d(x,y)-\epsilon R\big)^2\right\},
\end{eqnarray*}
which is bounded above by
\begin{eqnarray*}
&&N\omega_N (4\pi(1+\epsilon))^{-\frac{N}{2}}\int_R^\infty  r^{N-1}\left[\exp\left\{-\frac{r^2}{4(1+\epsilon)t}\right\}- \exp\left\{-\frac{(1+\delta(\delta+2)^2)r^2}{4(1+\epsilon)t}\right\}\right]\,\d r\\
&&+ N\omega_N (4\pi(1+\epsilon))^{-\frac{N}{2}}\int_R^\infty  r^{N-1}\exp\left\{-\frac{(1+\delta(\delta+2)^2)r^2}{4(1+\epsilon)t}\right\}\left[1-\frac{\kappa_\infty}{\kappa_x(\epsilon r)}\right]\,\d r\\
&&+ (4\pi(1+\epsilon))^{-\frac{N}{2}}\frac{\omega_N}{\kappa_\infty}\mu(B(y,\epsilon R))\exp\left\{-\frac{(1+\delta(\delta+2)^2)}{4(1+\epsilon)t}\big(d(x,y)-\epsilon R\big)^2\right\},
\end{eqnarray*}
since $s(x,r)\geq Nr^{N-1}\kappa_\infty$ by Lemma \ref{lowerboundary}. Thus,
\begin{eqnarray}\label{eq-1}
&&\int_{B(y,\epsilon R)}p_{(1+\epsilon)t}(x,z)\,\d\mu(z)\cr
&\leq&N\omega_N (4\pi(1+\epsilon))^{-\frac{N}{2}}\int_R^\infty  r^{N-1}\exp\left\{-\frac{r^2}{4(1+\epsilon)t}\right\}\frac{\delta(\delta+2)^2 r^2}{4(1+\epsilon) t}\,\d r\cr
&& + N\omega_N (4\pi(1+\epsilon))^{-\frac{N}{2}}\max_{r\geq R}\left[1-\frac{\kappa_\infty}{\kappa_x(\delta r)}\right]\int_R^\infty  r^{N-1}
\exp\left\{-\frac{(1+\delta(\delta+2)^2)r^2}{4(1+\epsilon)t}\right\}\,\d r\cr
&& + (4\pi(1+\epsilon))^{-\frac{N}{2}}\frac{\omega_N}{\kappa_\infty}\mu(B(y,\epsilon R))\exp\left\{-\frac{(1+\delta(\delta+2)^2)}{4(1+\epsilon)t}\big(d(x,y)-\epsilon R\big)^2\right\}.
\end{eqnarray}

By the elementary identity
$$\int_0^\infty r^a e^{-\frac{r^2}{b}}2r\,\d r=\Gamma\Big(\frac{a}{2}+1\Big)b^{\frac{a}{2}+1},\quad\mbox{for any }a\geq0,\,\ b>0,$$
concerning the first term in the right hand side of \eqref{eq-1}, we have that
\begin{eqnarray}\label{eq-2}
&&(4\pi(1+\epsilon))^{-\frac{N}{2}}\int_R^\infty  r^{N-1}\exp\left\{-\frac{r^2}{4(1+\epsilon)t}\right\}\frac{ r^2}{4(1+\epsilon) t}\,\d r\cr
&\leq&C_2 \big(1+R^N(4\pi(1+\epsilon)t)^{-\frac{N}{2}}\big)\exp\left\{-\frac{R^2}{4(1+\epsilon)t}\right\},
\end{eqnarray}
for some positive constant $C_2$ depending only on $N$. Similarly, by the integration by parts, there exist positive constants $C_3$ and $C_4$ depending on $N$ such that
\begin{eqnarray}\label{eq-3}
&&(4\pi(1+\epsilon))^{-\frac{N}{2}}\int_R^\infty  r^{N-1}\exp\left\{-\frac{(1+\delta(\delta+2)^2)r^2}{4(1+\epsilon)t}\right\}\,\d r\cr
&=&(4\pi(1+\epsilon))^{-\frac{N}{2}}\frac{1}{N}\left[r^{N}\exp\left\{-\frac{(1+\delta(\delta+2)^2)r^2}{4(1+\epsilon)t}\right\}\right]_R^\infty+\cr
&&(4\pi(1+\epsilon))^{-\frac{N}{2}}\frac{1}{N}\int_R^\infty  r^{N}\exp\left\{-\frac{(1+\delta(\delta+2)^2)r^2}{4(1+\epsilon)t}\right\}\frac{ 2(1+\delta(\delta+2)^2)r}{4(1+\epsilon) t}\,\d r\cr
&\leq&C_3 (4\pi(1+\epsilon))^{-\frac{N}{2}}\int_R^\infty  r^{N-1}\exp\left\{-\frac{(1+\delta(\delta+2)^2)r^2}{4(1+\epsilon)t}\right\}\frac{ 2(1+\delta(\delta+2)^2)r}{4(1+\epsilon) t}\,\d r\cr
&\leq&C_4 \big(1+R^N(4\pi(1+\epsilon)t)^{-\frac{N}{2}}\big)\exp\left\{-\frac{R^2}{4(1+\epsilon)t}\right\}.
\end{eqnarray}

Substituting \eqref{eq-2} and \eqref{eq-3} into \eqref{eq-1}, we obtain that
\begin{eqnarray}\label{eq-4}
&&\int_{B(y,\epsilon R)}p_{(1+\epsilon)t}(x,z)\,\d\mu(z)\cr
&\leq&C_5 \omega_N \big(1+R^N(4\pi(1+\epsilon)t)^{-\frac{N}{2}}\big)(\delta+\alpha(\delta,R))\exp\left\{-\frac{R^2}{4(1+\epsilon)t}\right\}+\cr
&&  (4\pi(1+\epsilon))^{-\frac{N}{2}}\frac{\omega_N}{\kappa_\infty}\mu(B(y,\epsilon R))\exp\left\{-\frac{(1+\delta(\delta+2)^2)}{4(1+\epsilon)t}\big(d(x,y)-\epsilon R\big)^2\right\},
\end{eqnarray}
where $C_5$ is some positive constant depending on $N$ and
$$\alpha(\delta,R)=\sup_{r\geq R}\left[1-\frac{\kappa_\infty}{\kappa_x(\delta r)}\right].$$

Since $\kappa_x(r)$ monotonically decreases to $\kappa_\infty$ as $r$ increases to $\infty$, by \eqref{harnack} and \eqref{eq-4}, we have
\begin{eqnarray*}
&&p_t(x,y)\cr
&\leq&(1+\epsilon)^{\frac{N}{2}}\exp\left\{\frac{\epsilon^2R^2}{4\epsilon t}\right\}\mu(B(y,\epsilon R))^{-1}\int_{B(y,\epsilon R)}p_{(1+\epsilon)t}(x,z)\,\d\mu(z)\cr
&\leq&C_5\frac{\omega_N(1+\epsilon)^{\frac{N}{2}}}{\mu(B(y,\epsilon R))}\big(1+R^N(4\pi(1+\epsilon)t)^{-\frac{N}{2}}\big)(\delta+\alpha(\delta,R))\exp\left\{\frac{\epsilon R^2}{4t}-\frac{R^2}{4(1+\epsilon)t}\right\}\cr
&&+  (4\pi(1+\epsilon))^{-\frac{N}{2}}\frac{\omega_N}{\kappa_\infty}\exp\left\{\frac{\epsilon R^2}{4t}-\frac{(1+\delta(\delta+2)^2)}{4(1+\epsilon)t}\big(d(x,y)-\epsilon R\big)^2\right\},
\end{eqnarray*}
Note that $\mu(B(y,\epsilon R))\geq \kappa_\infty (\epsilon R)^N$ and $R=(1-\epsilon)d(x,y)\geq \epsilon(1-\epsilon)\sqrt{t}$. Thus,
\begin{eqnarray}\label{eq-5}
&& p_t(x,y)\cr
&\leq&C_5\frac{\omega_N}{\kappa_\infty}\left[\frac{(1+\epsilon)^{\frac{N}{2}}}{ t^{N/2} \epsilon^{2N}(1-\epsilon)^{N}}+\frac{(4\pi t)^{-N/2}}{ \epsilon^N}\right](\delta+\alpha(\delta,R))\times \cr
&&\exp\left\{-\frac{(1-\epsilon-\epsilon^2)(1-\epsilon)^2}{4(1+\epsilon)t}d(x,y)^2\right\} + \cr
&&  (4\pi t)^{-\frac{N}{2}}\frac{\omega_N}{\kappa_\infty}
\exp\left\{\left[\frac{\epsilon(1-\epsilon)^2}{4t}-\frac{(1+\delta(\delta+2)^2)(1-\epsilon+\epsilon^2)^2}{4(1+\epsilon)t}\right]d(x,y)^2\right\}.
\end{eqnarray}
Choosing $\delta=\epsilon^{2N+1}$ and letting
$$\beta=\epsilon^{-2N}\alpha(\delta,(1-\epsilon)d(x,y)),$$
we deduce from \eqref{eq-5} that
\begin{eqnarray}\label{eq-6}
 p_t(x,y)\leq (1+C_6(\epsilon+\beta))\frac{\omega_N}{\kappa_\infty}(4\pi t)^{-N/2}\exp\left\{-\frac{1-\epsilon}{4t}d(x,y)^2\right\},
\end{eqnarray}
for some positive constant $C_6$ depending only on $N$.

Finally, combining \eqref{near} and \eqref{eq-6}, we finish the proof of the upper bound \eqref{UE}.
\end{proof}

\begin{remark}
An immediate observation is that
$$\lim_{d(x,y)\rightarrow\infty}\beta=0,$$
where $\beta$ is defined by \eqref{beta}.
\end{remark}

Immediate applications of Theorem \ref{main} are presented in the following two lemmas. The proofs are sketched since they follow the ones in the Riemannian setting; see \cite[Section 2]{LiTamWang}. The first one is on the sharp bounds for the (minimal) Green function, defined by
$$G(x,y)=\int_0^\infty p_t(x,y)\,\d t,$$
for any $x,y\in X$.
\begin{corollary}\label{greenbounds}
Let $(X,d,\mu)$ be an $\RCD(0,N)$ space with $N\in \mathbb{N}\setminus\{1,2\}$ having the maximum volume growth \eqref{max-vol-1}. Then, for any $\epsilon>0$, the Green function satisfies the estimate
\begin{eqnarray*}\label{green-bound}
(1+\epsilon(\epsilon+2)^2)^{1-\frac{N}{2}}\frac{d(x,y)^{2-N}}{N(N-2)\kappa_x(\epsilon d(x,y))}&\leq& G(x,y)\\
&\leq& (1+C_N(\epsilon+\beta))(1-\epsilon)^{1-\frac{N}{2}}\frac{d(x,y)^{2-N}}{N(N-2)\kappa_\infty},
\end{eqnarray*}
for any $y\in X$; in particular,
$$\lim_{d(x,y)\rightarrow\infty}\frac{N(N-2)G(x,y)}{d(x,y)^{N-2}}=\frac{1}{\kappa_\infty}.$$
\end{corollary}
\begin{proof}
The estimate follows from the elementary equality: for any $\lambda>0$ and $y\in X$,
\begin{eqnarray*}
\int_0^\infty t^{-\frac{N}{2}}\exp\left[-\frac{\lambda d(x,y)^2}{4t}\right]\,\d t&=&\left(\frac{4}{\lambda}\right)^{\frac{N}{2}-1}d(x,y)^{2-N}\int_0^\infty s^{\frac{N}{2}-2}e^{-s}\,\d t\\
&=&\left(\frac{4}{\lambda}\right)^{\frac{N}{2}-1}\Gamma\left(\frac{N}{2}-1\right)d(x,y)^{2-N}\\
&=&\frac{4^{N-1}\pi^{\frac{N}{2}}}{N(N-2)\lambda^{\frac{N}{2}-1}\omega_N}d(x,y)^{2-N},
\end{eqnarray*}
where the last inequality follows from the fact that $\omega_N=\pi^{N/2}/\Gamma(N/2+1)$ and $\Gamma(N/2+1)=\Gamma(N/2)N/2=\Gamma(N/2-1)N(N-2)/4$. The proof is completed by the integration of heat kernel lower and upper bounds in Theorem \ref{main} from 0 to $\infty$ along the time direction.
\end{proof}

The second application is on the large-time asymptotics of the heat kernel, which is a strengthening of the result obtained recently by the author with R. Jiang and H. Zhang in \cite[Theorem 4.1]{JLZ2014}.
\begin{corollary}\label{largetime}
Let $(X,d,\mu)$ be an $\RCD(0,N)$ space with $N\in \mathbb{N}\setminus\{1\}$ having the maximum volume growth \eqref{max-vol-1}. Then, for every path $t\mapsto y(t)$ from $(0,\infty)$ to $M$ satisfying $d(x,y(t))=O(\sqrt{t})$ as $t\rightarrow\infty$, it holds that
$$\lim_{t\rightarrow\infty}\mu(B(x,\sqrt{t}))p_t(x,y(t))\exp\left\{\frac{d(x,y(t))^2}{4t}\right\}=(4\pi)^{-\frac{N}{2}}\omega_N.$$
\end{corollary}
\begin{proof}
The case when $d(x,y(t))=o(\sqrt{t})$ as $t\rightarrow\infty$ follows from \cite[Theorem 4.1]{JLZ2014}, and the other case follows from the same method in the proof of \cite[Corollary 2.3]{LiTamWang} with the heat kernel bounds in Theorem \ref{main}.
\end{proof}

\section{Large-time asymptotics of the entropy}
\hskip\parindent In this section, we show the large-time asymptotics of entropies as our main application of Theorem \ref{main}. We first give definitions of the Perelman entropy and the Nash entropy in our non-smooth context. The former is introduced by G. Perelman as $\mathcal{W}$-entropy in his celebrated paper \cite{Perelman2002}, which turned out to be an important tool in the study of the Ricci flow. Our definition is motivated from Ni \cite{Ni2004a,Ni2004b,Ni2010}, where the similar entropy for the linear heat equation in the Riemannian manifold is studied. See also \cite{Lixd2012} for parallel studies on the linear heat equation with Laplacian replaced by Witten--Laplacian. The later is originated from J. Nash's seminal paper \cite{Nash1958}.

Let $(X,d,\mu)$ be an $\RCD(K,N)$ space with $K\in \R$ and $N\in (1,\infty)$, and let $t>0$. We always fix $x\in X$. Define the Perelman entropy as
$$\mathcal{W}(p,t)=\int_X\left(t|\nabla f|_w^2+f-N \right)p\,\d\mu,$$
where $f$ is defined by $p=(4\pi t)^{-N/2}e^{-f}$. And define the Nash entropy as
$$\mathcal{N}(p,t)=-\int_X p\log p\,\d\mu-\frac{N}{2}\log(4\pi t)-\frac{N}{2}.$$

The main result in this section is presented in the next theorem.
\begin{theorem}\label{largeentropy}
Let $(X,d,\mu)$ be an $\RCD(0,N)$ space with $N\in \mathbb{N}\setminus\{1\}$, having the maximum volume growth \eqref{max-vol-1}. Then
$$\lim_{t\rightarrow\infty}\mathcal{W}(p,t)=\lim_{t\rightarrow\infty}\mathcal{N}(p,t)=\log\Big(\frac{\kappa_\infty}{\omega_N}\Big).$$
\end{theorem}

In the following two subsections, we show the proof of the theorem.

\subsection{Large-time asymptotics of the Nash entropy}
Now we present the result on the large-time asymptotics of the Nash entropy. The method of proof, originated from \cite{Ni2010}, is a direct application of the sharp heat kernel bounds presented in Theorem \ref{main}.
\begin{theorem}\label{asym-entropy}
Let $(X,d,\mu)$ be an $\RCD(0,N)$ space with $N\in \mathbb{N}\setminus\{1\}$, having the maximum volume growth \eqref{max-vol-1}. Then
$$\lim_{t\rightarrow\infty}\mathcal{N}(p,t)=\log\Big(\frac{\kappa_\infty}{\omega_N}\Big).$$
\end{theorem}
\begin{proof}
Let $t>0$. Note that $\lim_{s\rightarrow\infty}\kappa_x(s)=\kappa_\infty$ and $\lim_{d(x,y)\rightarrow\infty}\beta=0$. Then, for any $\sigma>0$, there is a sufficient big constant $D$ such that, for any $y\in X$ satisfying $d(x,y)\geq D$, we have
\begin{eqnarray}\label{longtime-1}
\kappa_x(\epsilon^{2N+1}d(x,y))\leq (1+\sigma)\kappa_\infty,
\end{eqnarray}
 and $\beta\leq\sigma$, $\kappa_x(\epsilon d(x,y))\geq (1-\sigma)/\kappa_\infty$, where $\epsilon>0$. Hence,
\begin{eqnarray}\label{U.E.}
p_t(x,y)
\leq\left[1+C_N(\epsilon+\sigma)\right]\frac{\omega_N}{\kappa_\infty}(4\pi t)^{-\frac{N}{2}}\exp\left\{-\frac{1-\delta}{4t}d(x,y)^2\right\},
\end{eqnarray}
and
\begin{eqnarray}\label{L.E.}
p_t(x,y)\geq\frac{\omega_N}{\kappa_\infty}(1-\sigma)(4\pi t)^{-\frac{N}{2}}\exp\left\{-\frac{1+\epsilon(\epsilon+2)^2}{4t}d(x,y)^2\right\}.
\end{eqnarray}

We first show that
\begin{eqnarray}\label{uppernash}
\lim_{t\rightarrow\infty}\mathcal{N}(p,t)\leq\log\Big(\frac{\kappa_\infty}{\omega_N}\Big).
\end{eqnarray}
By the definition of $\mathcal{N}(p,t)$, the lower bound \eqref{LE}, and the stochastic completeness \eqref{stocom}, we have
\begin{eqnarray}\label{nash}
&&\mathcal{N}(p,t)\cr
&\leq& -\int_X p_t(x,y)\log\left[\frac{\omega_N}{\kappa_x(\epsilon d(x,y))}(4\pi t)^{-\frac{N}{2}}\exp\left\{-\frac{1+\epsilon(\epsilon+2)^2}{4t}d(x,y)^2\right\}\right]\,\d\mu(y)\cr
&& -\frac{N}{2}\log(4\pi t)-\frac{N}{2}\cr
&=&-\frac{N}{2}-\int_X p_t(x,y)\log\left[\frac{\omega_N}{\kappa_x(\epsilon d(x,y))}\right]\,\d\mu(y)\cr
&&+\frac{1+\epsilon(\epsilon+2)^2}{4t}\int_X p_t(x,y)d(x,y)^2\,\d\mu(y)\cr
&=:&-\frac{N}{2} + \rm{I}+\rm{II},
\end{eqnarray}
where we used the stochastic completeness \eqref{stocom} in the first equality.

Applying Lemma \ref{co-area} and \eqref{L.E.}, we obtain that
\begin{eqnarray*}
{\rm{I}}&\leq& -\int_0^\infty \frac{\omega_N}{\kappa_\infty}(1-\sigma)(4\pi t)^{-\frac{N}{2}}\exp\left\{-\frac{1+\epsilon(\epsilon+2)^2}{4t}r^2\right\}\log\left[\frac{\omega_N}{\kappa_x(\epsilon r)}\right]s(x,r)\,\d r\\
&=&-\left(\int_0^D+\int_D^\infty\right)\frac{\omega_N}{\kappa_\infty}(1-\sigma)(4\pi t)^{-\frac{N}{2}}\exp\left\{-\frac{1+\epsilon(\epsilon+2)^2}{4t}r^2\right\}\\
&&\times \log\left[\frac{\omega_N}{\kappa_x(\epsilon r)}\right]s(x,r)\,\d r\\
&=:&\rm{I}_1+\rm{I}_2.
\end{eqnarray*}
It is easy to know that
$$\lim_{t\rightarrow\infty}{\rm{I}_1}=0.$$
By Lemma \ref{lowerboundary} and \eqref{longtime-1}, we derive that
\begin{eqnarray*}
{\rm{I}_2}\leq \log\left[\frac{(1+\sigma)\kappa_\infty}{\omega_N}\right](1-\sigma)N\omega_N(4\pi t)^{-\frac{N}{2}}\int_D^\infty\exp\left\{-\frac{1+\epsilon(\epsilon+2)^2}{4t}r^2\right\}r^{N-1}\,\d r.
\end{eqnarray*}
By direct calculation, we get
$$\int_0^\infty\exp\left\{-\frac{1+\epsilon(\epsilon+2)^2}{4t}r^2\right\}r^{N-1}\,\d r=\frac{\Gamma(\frac{N}{2})}{2\pi^{\frac{N}{2}}}(4\pi t)^{\frac{N}{2}}\big(1+\epsilon(\epsilon+2)^2\big)^{-\frac{N}{2}}.$$
Hence,
\begin{eqnarray*}
{\rm{I}_2}&\leq& \log\left[\frac{(1+\sigma)\kappa_\infty}{\omega_N}\right](1-\sigma)N\omega_N(4\pi t)^{-\frac{N}{2}}\frac{\Gamma(\frac{N}{2})}{2\pi^{\frac{N}{2}}}(4\pi t)^{\frac{N}{2}}\big(1+\epsilon(\epsilon+2)^2\big)^{-\frac{N}{2}}\\
&=&\log\left[\frac{(1+\sigma)\kappa_\infty}{\omega_N}\right](1-\sigma)\big(1+\epsilon(\epsilon+2)^2\big)^{-\frac{N}{2}}.
\end{eqnarray*}
Thus,
\begin{eqnarray}\label{upper-I}
\lim_{t\rightarrow\infty}{\rm{I}}\leq\log\left[\frac{(1+\sigma)\kappa_\infty}{\omega_N}\right]
(1-\sigma)\big(1+\epsilon(\epsilon+2)^2\big)^{-\frac{N}{2}}.
\end{eqnarray}

Applying the upper bound \eqref{U.E.} and Lemma \ref{co-area}, we have
\begin{eqnarray*}
{\rm{II}}&\leq& \left[1+C_N(\epsilon+\sigma)\right]\frac{\big(1+\epsilon(\epsilon+2)^2\big)\omega_N}{4t(4\pi t)^{N/2}\kappa_\infty}\int_X \exp\left\{-\frac{1-\epsilon}{4t}d(x,y)^2\right\} d(x,y)^2\,\d\mu(y)\\
&=&\left[1+C_N(\epsilon+\sigma)\right]\frac{\big(1+\epsilon(\epsilon+2)^2\big)\omega_N}{4t(4\pi t)^{N/2}\kappa_\infty}\left(\int_{d(x,y)< D}+\int_{d(x,y)\geq D}\right) \exp\left\{-\frac{1-\epsilon}{4t}d(x,y)^2\right\}\\
&&\times  d(x,y)^2\,\d\mu(y)\\
&\leq&\left[1+C_N(\epsilon+\sigma)\right]\frac{\big(1+\epsilon(\epsilon+2)^2\big)\omega_N}{4t(4\pi t)^{N/2}\kappa_\infty}\left(D^2\mu(B(x,D))+\int_D^\infty \exp\left\{-\frac{1-\epsilon}{4t}r^2\right\} r^2s(x,r)\,\d r\right)\\
&=:&\left[1+C_N(\epsilon+\sigma)\right]\left(\rm{II}_1+\rm{II}_2\right).
\end{eqnarray*}
It is obvious to see that
$$\lim_{t\rightarrow\infty}{\rm{II}_1}=0.$$
Note that we can also require that $s(x,r)\leq(1+\sigma)N\kappa_\infty r^{N-1}$ for any $r>D$. Then
\begin{eqnarray*}
{\rm{II}_2}&\leq& \frac{\big(1+\epsilon(\epsilon+2)^2\big)\omega_N}{4t(4\pi t)^{N/2}\kappa_\infty}\int_D^\infty \exp\left\{-\frac{1-\epsilon}{4t}r^2\right\} r^2 (1+\sigma)N\kappa_\infty r^{N-1}\,\d r\\
&=&\frac{(1+\sigma)\big(1+\epsilon(\epsilon+2)^2\big)N\omega_N}{4t(4\pi t)^{N/2}}\int_D^\infty \exp\left\{-\frac{1-\epsilon}{4t}r^2\right\} r^{N+1}\,\d r\\
&\leq& \frac{(1+\sigma)\big(1+\epsilon(\epsilon+2)^2\big)N\omega_N}{4t(4\pi t)^{N/2}}\int_0^\infty \exp\left\{-\frac{1-\epsilon}{4t}r^2\right\} r^{N+1}\,\d r\\
&=& \frac{(1+\sigma)\big(1+\epsilon(\epsilon+2)^2\big)N\omega_N}{4t(4\pi t)^{N/2}}\frac{1}{2}\Gamma\Big(\frac{N}{2}+1\Big)\left(\frac{4t}{1-\epsilon}\right)^{\frac{N}{2}+1}\\
&=&\frac{N}{2}\frac{(1+\sigma)\big(1+\epsilon(\epsilon+2)^2\big)}{(1-\epsilon)^{N/2+1}}.
\end{eqnarray*}
Hence,
\begin{eqnarray}\label{upper-II}
\lim_{t\rightarrow\infty}{\rm{II}}\leq \frac{N}{2}\frac{(1+\sigma)\big(1+\epsilon(\epsilon+2)^2\big)}{(1-\epsilon)^{N/2+1}}.
\end{eqnarray}

Thus, combining \eqref{nash}, \eqref{upper-I} and \eqref{upper-II}, we obtain that
\begin{eqnarray*}\lim_{t\rightarrow\infty}\mathcal{N}(p,t)&\leq& -\frac{N}{2}+ \log\left[\frac{(1+\sigma)\kappa_\infty}{\omega_N}\right](1-\sigma)\big(1+\epsilon(\epsilon+2)^2\big)^{-\frac{N}{2}}\\
&&+\frac{N}{2}\frac{(1+\sigma)\big(1+\epsilon(\epsilon+2)^2\big)}{(1-\epsilon)^{N/2+1}}.
\end{eqnarray*}
Letting first $\sigma\rightarrow0$ and then $\epsilon\rightarrow0$, we finish the proof of \eqref{uppernash}.

Now we begin to show that
\begin{eqnarray}\label{lower-nash}
\lim_{t\rightarrow\infty}\mathcal{N}(p,t)\geq \log\Big(\frac{\kappa_\infty}{\omega_N}\Big).
\end{eqnarray}
By the definition of $\mathcal{N}(p,t)$ and the upper bound \eqref{U.E.}, we have that
\begin{eqnarray}\label{lowernash}
&&\mathcal{N}(p,t)\cr
&\geq&-\int_X p_t(x,y)\log\left[\big(1+C_N(\epsilon+\sigma)\big)\frac{\omega_N}{\kappa_\infty}(4\pi t)^{-\frac{N}{2}}\exp\left\{-\frac{1-\epsilon}{4t}d(x,y)^2\right\}\right]\,\d\mu(y)\cr
&&-\frac{N}{2}\log(4\pi t)-\frac{N}{2}\cr
&=&\frac{1-\epsilon}{4t}\int_X p_t(x,y)d(x,y)^2\,\d\mu(y)-\log\left[\big(1+C_N(\epsilon+\sigma)\big)\frac{\omega_N}{\kappa_\infty}\right]-\frac{N}{2},
\end{eqnarray}
where we used \eqref{stocom} again in the last equality. Let
$${\rm{III}}=\frac{1-\epsilon}{4t}\int_X p_t(x,y)d(x,y)^2\,\d\mu(y).$$
Then, by the lower bound \eqref{L.E.}, we deduce that
\begin{eqnarray*}
{\rm{III}}&\geq& \frac{1-\epsilon}{4t}\int_X \frac{\omega_N}{\kappa_\infty}(1-\sigma)(4\pi t)^{-\frac{N}{2}}\exp\left\{-\frac{1+\epsilon(\epsilon+2)^2}{4t}d(x,y)^2\right\}d(x,y)^2\,\d\mu(y)\\
&=&\frac{(1-\epsilon)(1-\sigma)\omega_N}{4t(4\pi t)^{N/2}\kappa_\infty}\int_X \exp\left\{-\frac{1+\epsilon(\epsilon+2)^2}{4t}d(x,y)^2\right\}d(x,y)^2\,\d\mu(y)\\
&=& \rm{III}_1+\rm{III}_2,
\end{eqnarray*}
where
\begin{eqnarray*}
{\rm{III}_1}=\frac{(1-\epsilon)(1-\sigma)\omega_N}{4t(4\pi t)^{N/2}\kappa_\infty}\int_{d(x,y)<D} \exp\left\{-\frac{1+\epsilon(\epsilon+2)^2}{4t}d(x,y)^2\right\}d(x,y)^2\,\d\mu(y),
\end{eqnarray*}
and
\begin{eqnarray*}
{\rm{III}_2}=\frac{(1-\epsilon)(1-\sigma)\omega_N}{4t(4\pi t)^{N/2}\kappa_\infty}\int_{d(x,y)\geq D} \exp\left\{-\frac{1+\epsilon(\epsilon+2)^2}{4t}d(x,y)^2\right\}d(x,y)^2\,\d\mu(y).
\end{eqnarray*}
It is easy to see that
\begin{eqnarray}\label{III-1}
\lim_{t\rightarrow\infty}{\rm{III}_1}=0.
\end{eqnarray}
By Lemma \ref{co-area}, we have that
\begin{eqnarray*}
{\rm{III}_2}=\frac{(1-\epsilon)(1-\sigma)\omega_N}{4t(4\pi t)^{N/2}\kappa_\infty}\int_D^\infty \exp\left\{-\frac{1+\epsilon(\epsilon+2)^2}{4t}r^2\right\}r^2s(x,r)\,\d r.
\end{eqnarray*}
Applying Lemma \ref{lowerboundary} again, we derive that
\begin{eqnarray*}
{\rm{III}_2}&\geq&\frac{(1-\epsilon)(1-\sigma)\omega_N}{4t(4\pi t)^{N/2}\kappa_\infty}\int_D^\infty \exp\left\{-\frac{1+\epsilon(\epsilon+2)^2}{4t}r^2\right\}r^2N\kappa_\infty r^{N-1}\,\d r\\
&=&\frac{(1-\epsilon)(1-\sigma)N\omega_N}{4t(4\pi t)^{N/2}}\int_D^\infty \exp\left\{-\frac{1+\epsilon(\epsilon+2)^2}{4t}r^2\right\}r^{N+1}\,\d r\\
&=&\frac{(1-\epsilon)(1-\sigma)N\omega_N}{4t(4\pi t)^{N/2}}\left(\int_0^\infty -\int_0^D\right) \exp\left\{-\frac{1+\epsilon(\epsilon+2)^2}{4t}r^2\right\}r^{N+1}\,\d r\\
&\geq&\frac{(1-\epsilon)(1-\sigma)N\omega_N}{4t(4\pi t)^{N/2}}\left[\frac{1}{2}\Gamma\Big(\frac{N}{2}+1\Big)\left(\frac{4t}{1+\epsilon(\epsilon+2)^2}\right)^{\frac{N}{2}+1}-\frac{D^{N+2}}{N+2}\right]\\
&=&\frac{N}{2}\frac{(1-\epsilon)(1-\sigma)}{\big(1+\epsilon(\epsilon+2)^2)^{N/2+1}} -  \frac{(1-\epsilon)(1-\sigma)N\omega_N}{4t(4\pi t)^{N/2}}\frac{D^{N+2}}{N+2}.
\end{eqnarray*}
Hence,  it is clear that
\begin{eqnarray}\label{III-2}
\lim_{t\rightarrow\infty}{\rm{III}_2}\geq \frac{N}{2}\frac{(1-\epsilon)(1-\sigma)}{\big(1+\epsilon(\epsilon+2)^2)^{N/2+1}}.
\end{eqnarray}
Thus, combining \eqref{lowernash} with \eqref{III-1} and \eqref{III-2}, we finish the proof of the lower estimate \eqref{lower-nash}.
\end{proof}

\subsection{Large-time asymptotics of the Perelman entropy}
In this subsection, we show the large-time asymptotics of the Perelman entropy, which is presented in the following theorem.
\begin{theorem}\label{asym-P-entropy}
Let $(X,d,\mu)$ be an $\RCD(0,N)$ space with $N\in \mathbb{N}\setminus\{1\}$, having the maximum volume growth \eqref{max-vol-1}.
Then
$$\lim_{t\rightarrow\infty}\mathcal{W}(p,t)=\lim_{t\rightarrow\infty}\mathcal{N}(p,t).$$
\end{theorem}

We should mention that the method of proof of Theorem \ref{asym-P-entropy} is originated from \cite{Ni2004a,Ni2004b}. Thus we first need to prove the monotonicity of the $\mathcal{W}$ functional with respect to the time variable. In a very recent manuscript \cite{JZ2015}, R. Jiang and H. Zhang proved the monotonicity in the case when the metric measure space $(X,d,\mu)$ is compact by a different method.

Define
$$W_t=t|\nabla\log p_t|_w^2-2t\frac{\Delta p_t}{p_t}-\log p_t-\frac{N}{2}\log t.$$
Hence, it is immediate to see that
$$\mathcal{W}(p,t)=\int_Xp_{t}W_{t}\,\d\mu-\frac{N}{2}\log(4\pi)-N.$$

\begin{theorem}[Monotonicity]\label{mono-W}
Let $(X,d,\mu)$ be an $\RCD(0,N)$ space with $N\in (1,\infty)$, and let $T>0$. For any $t\in [0,T)$, it holds that
$$\mathcal{W}(p,T-t)\geq \mathcal{W}(p,T).$$
\end{theorem}

In order to prove Theorem \ref{mono-W}, we need the following lemmas. Let $\delta>0$ and let $0\leq f\in L^1(X)\cap L^\infty(X)$. Set $f_\delta=f+\delta$. For every $t\in [0,T]$, define
$$\Psi(t)=P_t(P_{T-t}f_\delta|\nabla\log P_{T-t}f_\delta|_w^2).$$
\begin{lemma}\label{lemma-1}
Let $(X,d,\mu)$ be an $\RCD(K,N)$ space with $K\in\R$ and $N\in (1,\infty)$, and let $T>0$ and $\delta>0$. Suppose that $0\leq f\in L^1(X)\cap L^\infty(X)$. Then, for any $t\in[0,T)$, it holds that
\begin{eqnarray*}
\int_0^t\Psi(s)\,\d s=P_t(P_{T-t}f_\delta\log P_{T-t}f_\delta)-P_Tf_\delta\log P_Tf_\delta.
\end{eqnarray*}
\end{lemma}
\begin{proof}
Let $T>0$ and $\delta>0$. Take nonnegative functions $f$ from $L^1(X)\cap L^\infty(X)$. Let $\theta(s)=(s+\delta)\log(s+\delta)-(1+\log\delta)s-\delta\log\delta$, $s\geq0$. Then $\theta(0)=0$, $\theta'(s)=\log(s+\delta)-\log\delta$, $\theta''(s)=\frac{1}{s+\delta}\in(0,1/\delta]$. Let $t\in[0,T)$. Since $P_{T-t}f\in \mathcal{D}(\Delta)$, by \cite[Corollary 6.1.4]{BouleauHirsch1991}, $\theta(P_{T-t}f)\in \mathcal{D}(\Delta_1)$, where $\Delta_1$ is the smallest closed extension of the generator of $\{P_t\}_{t\geq0}$ restricted to $\{f\in \mathcal{D}(\Delta)\cap L^1(X,\mu):\,\Delta f\in L^1(X,\mu)\}$.

Hence, for $\L^1$-a.e. $t\in [0,T)$, we deduce that
\begin{eqnarray*}
&&\frac{\d}{\d t}P_t(P_{T-t}f_\delta\log P_{T-t}f_\delta)\\
&=&\frac{\d}{\d t}P_t\left[\theta(P_{T-t}f)+(1+\log\delta)P_{T-t}f+\delta\log\delta \right]\\
&=&\frac{\d}{\d t} P_t\left[\theta(P_{T-t}f)\right]\\
&=&P_t\left[\Delta_1\theta(P_{T-t}f)-\theta'(P_{T-t}f)\Delta P_{T-t}f\right]\\
&=&P_t\left[\big(\log(P_{T-t}f+\delta)-\log\delta\big)\Delta P_{T-t}f+\frac{|\nabla P_{T-t}f|^2_w}{P_{T-t}f+\delta}\right]-P_t\big(\theta'(P_{T-t}f)\Delta P_{T-t}f\big)\\
&=&P_t(P_{T-t}f_\delta|\nabla\log P_{t-t}f_\delta|_w^2)=\Psi(t),
\end{eqnarray*}
where we have used the fact that $\Delta_1(P_tg)=P_t(\Delta_1g)$ for any $g\in\mathcal{D}(\Delta_1)$ and $t>0$ in the third equality and \cite[Corollary 6.1.4]{BouleauHirsch1991} again in the forth equality.

Integrating both sides on [0,t] with respect to $\d t$, we complete the proof.
\end{proof}

The next one is borrowed from \cite[Proposition 5.2]{jia14} (see also \cite{GarofaloMondino} for the particular case when $\mu$ is a probability measure).
\begin{lemma}\label{lemma-2}
Let $(X,d,\mu)$ be an $\RCD(K,N)$ space with $K\in\R$ and $N\in (1,\infty)$, and let $T>0$ and $\delta>0$. Suppose that $0\leq f,\psi\in L^1(X)\cap L^\infty(X)$. Let $a\in C^1([0,T],[0,\infty))$ and $\gamma\in C([0,T],\R)$. Then, for a.e. $t\in[0,T]$, it holds that
\begin{eqnarray*}
&&\frac{\d}{\d t}\int_X a(t)\Psi(t)\psi\,\d\mu\\
&\geq&\int_X\left[\left(a'(t)-\frac{4a(t)\gamma(t)}{N}+2Ka(t)\right)\Psi(t)+\frac{4a(t)\gamma(t)}{N}\Delta P_Tf_\delta-\frac{2a(t)\gamma(t)^2}{N}P_Tf_\delta\right]\psi\,\d\mu.
\end{eqnarray*}
\end{lemma}

\begin{proof}[Proof of Theorem \ref{mono-W}]
The method used here follows essentially the proof of \cite[Proposition 2.6]{BG2011} in the context of complete and smooth Riemannian manifolds. For any $t\in [0,T)$, let $a(t)=T-t$ and $\gamma(t)=-\frac{N}{2(T-t)}$. On the one hand, applying Lemma \ref{lemma-2} with $P_{T-t}f_\delta$ replaced by $p_{T-t}$ and integrating on [0,T] with respect to $\d t$, we have
\begin{eqnarray*}
&&\int_X\Psi(t)a(t)\psi\,\d\mu-\int_X\Psi(0)a(0)\psi\,\d\mu\\
&\geq&\int_0^t\int_X\left[\big(1+2K(T-s)\big)\Psi(s)-2\Delta p_T -\frac{N}{2(T-s)}p_T\right]\psi\,\d\mu\d s\\
&=&\int_0^t\int_X\Psi(s)\psi\,\d\mu\d s+ 2K\int_0^t\int_X(T-s)\Psi(s)\psi\,\d\mu\d s-2t\int_X(\Delta p_T)\psi\,\d\mu\\
&&-\frac{N}{2}\int_0^T\int_X\frac{1}{T-s}p_T\psi\,\d\mu\d s.
\end{eqnarray*}
By Lemma \ref{lemma-1}, we deduce that
\begin{eqnarray*}
&&(T-t)\int_X\Psi(t)\psi\,\d\mu\\
&\geq& T\int_X p_T|\nabla\log p_T|_w^2\psi\,\d\mu + \int_0^t\int_X \Psi(s)\psi\,\d\mu\d s-2t\int_X (\Delta p_T)\psi\,\d\mu\\
&&-\frac{N}{2}\log T\int_X p_T\psi\,\d\mu + \frac{N}{2}\log(T-t)\int_X p_T\psi\,\d\mu\\
&=&T\int_Xp_T|\nabla\log p_T|_w^2\psi\,\d\mu +\int_X(p_{T-t}\log p_{T-t})P_t\psi\,\d\mu-\int_X(p_T\log p_T )\psi\,\d\mu\\
&&-2T\int_X (\Delta p_T)\psi\,\d\mu +2(T-t)\int_X(\Delta p_T) \psi\d\mu-\frac{N}{2}\log T\int_Xp_T\psi\,\d\mu\\
&&+\frac{N}{2}\log(T-t)\int_Xp_T\psi\,\d\mu\\
&=&\int_Xp_TW_T\psi\,\d\mu+\int_X(p_{T-t}\log p_{T-t})P_t\psi\,\d\mu+2(T-t)\int_X(\Delta p_T) \psi\d\mu.
\end{eqnarray*}

On the other hand,
\begin{eqnarray*}
&&\int_XP_t(p_{T-t}W_{T-t})\psi\,\d\mu\\
&=&\int_XP_t\left(p_{T-t}\left[(T-t)|\nabla\log p_{T-t}|_w^2-2(T-t)\frac{\Delta p_{T-t}}{p_{T-t}}-\log p_{T-t}-\frac{N}{2}\log(T-t)\right]\right)\psi\,\d\mu\\
&=&(T-t)\int_X \Psi(t)\psi\,\d\mu-2(T-t)\int_X\Delta P_t(p_{T-t})\psi\,\d\mu-\int_X P_t(p_{T-t}\log p_{T-t})\psi\,\d\mu\\
&&-\frac{N}{2}\log(T-t)\int_X P_t(p_{T-t})\psi\,\d\mu\\
&=&(T-t)\int_X \Psi(t)\psi\,\d\mu-2(T-t)\int_X(\Delta p_{T})\psi\,\d\mu-\int_X P_t(p_{T-t}\log p_{T-t})\psi\,\d\mu\\
&&-\frac{N}{2}\log(T-t)\int_X p_{T}\psi\,\d\mu.
\end{eqnarray*}

Thus, for any $t\in [0,T)$ and any $0\leq\psi\in L^1(X)\cap L^\infty(X)$, we have
$$\int_XP_t(p_{T-t}W_{T-t})\psi\,\d\mu\geq\int_Xp_TW_T\psi\,\d\mu.$$
By approximation argument, the above inequality also holds for any $0\leq\psi\in L^\infty(X)$.
Therefore, combining with the stochastic completeness, i.e., $P_t1=1$ for any $t>0$, we arrive at
$$\mathcal{W}(p,T-t)\geq\mathcal{W}(p,T),$$
which completes the proof.
\end{proof}

In \cite{Ni2004a}, L. Ni showed that if $M$ is a finite-dimensional complete Riemannian manifold with nonnegative Ricci curvature, then $M$ has maximum volume growth is equivalent to that the Perelman entropy $\mathcal{W}(p,t)$ has a lower bound. The next theorem generalize this nice result to the non-smooth setting. Note that we do not require $N\in \mathbb{N}$ in the next theorem.
\begin{theorem}\label{max-vol-W}
Let $(X,d,\mu)$ be an $\RCD(0,N)$ space with $N\in (1,\infty)$. Then $(X,d,\mu)$ has the maximum volume growth \eqref{max-vol-1} if and only if,  there exists a constant $A>0$ such that
$$\mathcal{W}(p,t)\geq-A,\quad\mbox{for any }t>0.$$
\end{theorem}
\begin{proof}
Suppose at first that \eqref{max-vol-1} holds. Then, for any $t>0$,
\begin{eqnarray*}
\mathcal{W}(p,t)&=&\int_X(t|\nabla f|^2_w+f-N)p\,\d\mu\\
&\geq&\int_X\left[-\log p_t(x,y)-\frac{N}{2}\log(4\pi t)-N\right]p_t(x,y)\,\d\mu(y)\\
&=&-\int_Xp_t(x,y)\log p_t(x,y)\,\d\mu(y)-\frac{N}{2}\log(4\pi t)-N,
\end{eqnarray*}
where $f=-\log p-\frac{N}{2}\log(4\pi t)$. Since $r^{-N}\mu(B(x,r))$ is monotonically decreasing as $r$ increasing, we derive from \eqref{max-vol-1} that $\mu(B(x,r))\geq \kappa_\infty r^N$ for all $r>0$. Then, by the heat kernel upper bound in \eqref{fullgaussian}, i.e.,
$$p_t(x,y)\leq \frac{C_0(N)}{\mu(B(x,\sqrt{t}))}\leq \frac{C_0(N)}{\kappa_\infty t^{N/2}},$$
we obtain that
$$\mathcal{W}(p,t)\geq -\log\Big(\frac{C_0(N)}{\kappa_\infty}\Big)-\frac{N}{2}\log(4\pi)-N.$$

Conversely, suppose that there exists a constant $A>0$ such that $\mathcal{W}(p,t)\geq-A$, for any $t>0$. By the Li--Yau inequality (see \cite{jia14}), i.e., for any $t>0$,
$$|\nabla_y \log p_t(x,\cdot)|^2_w-\partial_t \log p_t(x,\cdot)\leq \frac{N}{2t},\quad\mu\mbox{-a.e. in }X,$$
we have that
\begin{eqnarray}\label{W-1}
&&t\int_X |\nabla f|_w^2(y)p_t(x,y)\,\d\mu(y)=t\int_X |\nabla_y \log p_t(x,y)|_w^2p_t(x,y)\,\d\mu(y)\cr
&\leq& t\int_X\Big(\partial_t \log p_t(x,y)+\frac{N}{2t}\Big)p_t(x,y)\,\d\mu(y)=t\int_X\Delta_y p_t(x,y)\,\d\mu(y)+\frac{N}{2}\cr
&=&\frac{N}{2}.
\end{eqnarray}
 Here, the last equality holds due to that we can choose a sequence of  Lipschitz cut-off functions $\{\chi_n\}_{n\in \mathbb{N}}$ such that, for any $y\in X$,
\begin{equation*}
\chi_n(y)=
\begin{cases}
1,\quad &{\hbox{on}}\,\ B(o,2^n),\\
0,\quad &{\hbox{on}}\,\ X\setminus B(o,2^{n+1}),\\
2-2^{-n}d(y,o),\quad &{\hbox{on}}\,\ B(o,2^{n+1})\setminus B(o,2^{n}),
\end{cases}
\end{equation*}
for some $o\in X$, and $|\nabla \chi_n|_w\leq 2^{-n}$, and then
\begin{eqnarray*}
\left|\int_X\Delta_y p_t(x,y)\,\d\mu(y)\right|&=&\left|\lim_{n\rightarrow\infty}\int_X\chi_n(y)\Delta_y p_t(x,y)\,\d\mu(y)\right|\\
&=&\left|-\lim_{n\rightarrow\infty}2\int_X\langle\nabla\chi_n, \nabla_y p_t(x,\cdot)\rangle(y)\,\d\mu(y)\right|\\
&\leq&\lim_{n\rightarrow\infty}2\int_X|\nabla\chi_n|_w(y) |\nabla_yp_t(x,\cdot)|_w(y)\,\d\mu(y)\\
&\leq&\lim_{n\rightarrow\infty}\frac{1}{2^{n-1}}\int_X |\nabla_yp_t(x,\cdot)|_w(y)\,\d\mu(y)=0,
\end{eqnarray*}
where the last equality is implied by \cite[Corollary 1.1]{JLZ2014}.
By the lower bound of the heat kernel in \eqref{fullgaussian}, we obtain that
\begin{eqnarray*}
&&-\int_X p_t(x,y)\log p_t(x,y)\,\d\mu(y)\\
&\leq&-\int_X p_t(x,y)\log\left[\frac{1}{C(N)\mu(B(x,\sqrt{t}))}\exp\left\{-C(N)\frac{d(x,y)^2}{t}\right\}\right]\,\d\mu(y)\\
&=&\log C(N) + \log\mu(B(x,\sqrt{t})) + \frac{C(N)}{t}\int_X d(x,y)^2 p_t(x,y)\,\d\mu(y),
\end{eqnarray*}
where $C(N)$ is a positive constant depending on $N$. Let
$${\rm{J}}=\frac{C(N)}{t}\int_X d(x,y)^2 p_t(x,y)\,\d\mu(y).$$
By the upper bound of the heat kernel in \eqref{fullgaussian}, we deduce that
\begin{eqnarray*}
{\rm{J}}&\leq& \frac{C_0(N)}{t\mu(B(x,\sqrt{t}))}\int_X d(x,y)^2\exp\left\{-\frac{d(x,y)^2}{5t}\right\} \,\d\mu(y)\\
&=&\frac{C_0(N)}{t\mu(B(x,\sqrt{t}))}\left(\int_{d(x,y)< \sqrt{t}} + \int_{d(x,y)\geq \sqrt{t}}\right)d(x,y)^2\exp\left\{-\frac{d(x,y)^2}{5t}\right\} \,\d\mu(y)\\
&=:&\rm{J}_1+\rm{J}_2.
\end{eqnarray*}
It is easy to know that
$${\rm{J}}_1\leq C_1(N).$$
By splitting the region of integration into annular regions and by the doubling property in Lemma \ref{doubling}(i), we have that
\begin{eqnarray*}
{\rm{J}_2}&=&\frac{C_0(N)}{t\mu(B(x,\sqrt{t}))}\sum_{i=1}^\infty\int_{i\sqrt{t}\leq d(x,y)<(i+1)\sqrt{t}}  d(x,y)^2\exp\left\{-\frac{d(x,y)^2}{5t}\right\} \,\d\mu(y)\\
&\leq&\frac{C_0(N)}{t\mu(B(x,\sqrt{t}))}\sum_{i=1}^\infty (i+1)^2t\exp\left\{-\frac{i^2}{5}\right\}\mu(B(x,(i+1)\sqrt{t}))\\
&\leq& C_0(N)\sum_{i=1}^\infty  (i+1)^{N+2}\exp\left\{-\frac{i^2}{5}\right\}\\
&\leq& C_2(N).
\end{eqnarray*}
Hence, ${\rm{J}}\leq C_3(N)$. Thus, for any $t>0$,
\begin{eqnarray}\label{W-2}
-\int_X p_t(x,y)\log p_t(x,y)\,\d\mu(y)\leq C_4(N) + \log\mu(B(x,\sqrt{t})).
\end{eqnarray}
Combining \eqref{W-1} and \eqref{W-2} with the assumption, we have that, for any $t>0$,
$$-\frac{N}{2}\log(4\pi t) + C_5(N) + \log\mu(B(x,\sqrt{t})) \geq \mathcal{W}(p,t)\geq -A,$$
which immediately implies that
$$\mu(B(x,\sqrt{t}))\geq (4\pi)^{\frac{N}{2}}e^{-A-C_5(N)}t^{\frac{N}{2}},\quad\mbox{for any }t>0.$$
Therefore, we complete the proof.
\end{proof}
\begin{remark}
In fact, for $N\in \mathbb{N}\setminus\{1\}$, we can also use the heat kernel upper and lower bounds in Theorem \ref{main}, and deduce the same conclusion with different constants following the same method in the proof of Theorem \ref{asym-entropy}.
\end{remark}
\begin{corollary}\label{bound-W}

Let $(X,d,\mu)$ be an $\RCD(0,N)$ space with $N\in\mathbb{N}\setminus\{1\}$. Suppose that $(X,d,\mu)$ has the maximum volume growth \eqref{max-vol-1}. Then
$$\log\left(\frac{\kappa_\infty}{\omega_N}\right)\leq\mathcal{W}(p,t)\leq 0,\quad\mbox{for any }t>0.$$
\end{corollary}
\begin{proof}
The lower bound of $\mathcal{W}(p,t)$ can be obtained by a similarly method as the first part in the proof of Theorem \ref{max-vol-W} by using the sharp heat kernel upper bound \eqref{UE}. For the upper bound,
\begin{eqnarray*}
\mathcal{W}(p,t)&=&-\int_X p_t\log p_t\,\d\mu+t\int_X|\nabla\log p_t|_w^2p_t\,\d\mu-\frac{N}{2}\\
&\leq&t\int_X\left(|\nabla\log p_t|_w^2-\frac{N}{2t}\right)p_t\,\d\mu\leq0,
\end{eqnarray*}
where we used the stochastic completeness in the first inequality and \eqref{W-1} in the second inequality.
\end{proof}

Now we begin the proof of the large-time asymptotics of the Perelman entropy (see \cite[page 371]{Ni2004b}).
\begin{proof}[Proof of Theorem \ref{asym-P-entropy}]
By the assumption, from Corollary \ref{bound-W}, we have that $\mathcal{W}(p,t)$ is bounded uniformly in $t$. From Theorem \ref{asym-entropy}, we know that $\lim_{t\rightarrow\infty}\mathcal{N}(p,t)=\log\left(\kappa_\infty/\omega_N\right)$. Since
$$\mathcal{W}(p,t)=\mathcal{N}(p,t)+t\int_X|\nabla\log p_t|_w^2p_t\,\d\mu-\frac{N}{2},$$
and $\int_X|\nabla\log p_t|_w^2p_t\,\d\mu>0$, we know that there exists a real sequence $\{t_i\}_{i\geq 1}$, which tends to $\infty$ as $i\rightarrow\infty$, such that $\lim_{i\rightarrow\infty}t_i\int_X|\nabla\log p_{t_i}|_w^2p_{t_i}\,\d\mu-\frac{N}{2}=0$. Thus, by the monotonicity in Theorem \ref{mono-W}, we conclude that $\lim_{t\rightarrow\infty}\mathcal{W}(p,t)=\lim_{t\rightarrow\infty}\mathcal{N}(p,t)$.
\end{proof}

Finally, combining Theorems \ref{asym-entropy} and \ref{asym-P-entropy} together, we finish the proof of Theorem \ref{largeentropy}.

\subsection*{Acknowledgment}
\hskip\parindent  The author would like to thank Professor Bin Qian for advising him to study the Perelman entropy in metric measure spaces, and to thank Professors Dejun Luo and Yongsheng Song for a nice discussing when the author gave a talk on this topic in the Institute of Applied Mathematics, CAAS on November 20, 2015. This work started when the author was a research fellow in Macquarie University from October 2014 to October 2015. The author also would like to thank Professor Adam Sikora for his interest.

\end{document}